\documentclass[a4paper,12pt]{article}

\usepackage[a4paper]{geometry}
\usepackage{cmap}%
\usepackage[backend=bibtex,maxnames=10,defernumbers=true,giveninits=true,sortcites=true]{biblatex}
\usepackage[compact,small]{titlesec}
\usepackage[unicode,hyperfootnotes=false]{hyperref}
\usepackage{hyperxmp}
\usepackage{bookmark}

\hypersetup{%
        pdftitle={The non-linear sewing lemma II: Lipschitz continuous formulation},
        pdfauthor={Antoine Brault; Antoine Lejay},
        pdfdate={2021-01-27},
	pdfkeywords={rough paths, rough flows, non linear sewing lemma},
	pdflang={en},
	pdfversionid={R1},
}

\usepackage[T1]{fontenc}
\usepackage[utf8]{inputenc}
\usepackage{mathtools}
\usepackage{amsthm,amssymb}
\usepackage[cal=boondoxo]{mathalfa}
\usepackage{mathabx}
\usepackage{microtype}
\usepackage{enumitem}
\usepackage{csquotes}

\SetEnumitemKey{thm}{itemsep=0pt,partopsep=0pt,parsep=0pt,topsep=-\parskip,label={{\normalfont(\roman*)}}}

\newtheorem{proposition}{Proposition}
\newtheorem{theorem}{Theorem}

\newtheorem{lemma}{Lemma}
\newtheorem{corollary}{Corollary}
\theoremstyle{definition}
\newtheorem{definition}{Definition}
\newtheorem{notation}{Notation}

\theoremstyle{remark}
\newtheorem{remark}{Remark}
\newtheorem{example}{Example}

\newcommand{\itvoc}[1]{(#1]}
\newcommand{\itvco}[1]{[#1)}

\newcommand{\wcf}{f^{\dagger}}

\newcommand{\oldwidecheck}[1]{{#1}^{\dagger}}

\DeclarePairedDelimiterXPP{\normpa}[2]{}{\|}{\|}{_{#1,p}}{#2}
\DeclarePairedDelimiterXPP{\normf}[2]{}{\|}{\|}{_{#1}}{#2}
\DeclarePairedDelimiterXPP{\normsup}[1]{}{\|}{\|}{_{\infty}}{#1}
\DeclarePairedDelimiterXPP{\normp}[1]{}{\|}{\|}{_{p}}{#1}
\DeclarePairedDelimiterXPP{\normpp}[1]{}{\|}{\|}{_{\frac{p}{2}}}{#1}
\DeclarePairedDelimiterXPP{\generalnorm}[2]{}{\|}{\|}{_{#1}}{#2}
\DeclarePairedDelimiter{\simplenorm}{\|}{\|}
\DeclarePairedDelimiter{\eucnorm}{|}{|}
\DeclarePairedDelimiterXPP{\normhold}[2]{}{\|}{\|}{_{#1}}{#2}
\DeclarePairedDelimiterXPP{\norma}[2]{}{\|}{\|}{_{#1,1}}{#2}
\DeclarePairedDelimiterXPP{\normlip}[1]{}{\|}{\|}{_\mathrm{Lip}}{#1}
\DeclarePairedDelimiterXPP{\normlipl}[2]{}{\|}{\|}{_{\mathrm{Lip},#2}}{#1}
\DeclarePairedDelimiterX\Fam[1]\{\}{%
  \renewcommand\given{\SetSymbol[\delimsize]}#1}
\newcommand\given{\nonscript\:\delimsize\vert\nonscript\:\mathopen{}} 
\newcommand\SetSymbol[1][]{\nonscript\:#1\vert\nonscript\:\mathopen{}\allowbreak}
\DeclarePairedDelimiterX\Set[1]\{\}{%
  \renewcommand\given{\SetSymbol[\delimsize]}#1}

\newcommand{\coleq}{\mathbin{:=}}

\newcommand{\id}{\mathfrak{i}}

\newcommand{\vd}{\,\mathrm{d}}

\newcommand{\RR}{\mathbb{R}}

\newcommand{\TT}{\mathbb{T}}
\newcommand{\NN}{\mathbb{N}}
\newcommand{\bx}{\mathbf{x}}

\newcommand{\uU}{\mathrm{U}}
\newcommand{\uW}{\mathrm{W}}
\newcommand{\uT}{\mathrm{T}}
\newcommand{\uV}{\mathrm{V}}
\newcommand{\cpC}{\mathcal{C}^{p-\omega}}
\newcommand{\cpV}{\mathcal{C}^{p-\mathrm{var}}}
\newcommand{\cbV}{\mathcal{C}^{1-\mathrm{var}}}
\newcommand{\cC}{\mathcal{C}}
\newcommand{\cCb}{\mathcal{C}_{\mathrm{b}}}
\newcommand{\cT}{\mathcal{T}}
 \newcommand{\cA}{\mathcal{A}}
\newcommand{\cF}{\mathcal{F}}
\newcommand{\cG}{\mathcal{G}}
\newcommand{\cS}{\mathcal{S}}

\newcommand{\eqdef}{\mathbin{:=}}

\DeclarePairedDelimiter{\abs}{|}{|}

\newcommand{\rTT}{\mathbb{T}_+}
\newcommand{\rlTT}{\mathbb{T}_\pm}
\newcommand{\lTT}{\mathbb{T}_-}

\newcommand{\run}{\mathrm{I}}
\newcommand{\rdeux}{\mathrm{II}}
\newcommand{\rtrois}{\mathrm{III}}

\newcommand{\hf}{{\mathrlap{\protect\overrightarrow{f}}\hphantom{f}\,}}

\begin{document}

\title{The non-linear sewing lemma II: Lipschitz continuous formulation}
\author{Antoine Brault\thanks{Institut de Mathématiques de Toulouse, UMR 5219; Université de Toulouse, UPS IMT, F-31062 Toulouse Cedex 9, France}\thanks{Center for Mathematical Modeling (CNRS UMI 2807), University of Chile,      
\texttt{abrault@dim.uchile.cl}} \and Antoine Lejay\thanks{Université de Lorraine, CNRS, Inria, IECL, F-54000 Nancy, France, \texttt{antoine.lejay@univ-lorraine.fr}}}
\date{January 27, 2021}

\maketitle

\begin{abstract}
We give an unified framework to solve rough differential equations.
Based on flows, our approach unifies the former ones developed by
Davie, Friz-Victoir and Bailleul. The main idea is to build a flow from the
iterated product of an almost flow which can be viewed as a good approximation
of the solution at small time. In this second article,
we give tractable conditions under which the limit flow is Lipschitz continuous
and its links with uniqueness of solutions of rough differential equations. 
We also give perturbation formulas on almost flows
which link the former constructions.
\end{abstract}


\section{Introduction}

Rough paths theory was introduced to deal with differential equations
driven by an irregular deterministic path multidimensional $x$ of the type
\begin{align}
\label{eq:intro_rde}
y_t=a+\int_0^tf(y_r)\vd x_r,
\end{align}
where $a$ is an initial condition and $f$ a smooth function.
Typically, the irregularity of $x$ is measured in $\alpha$-Hölder ($\alpha\leq 1$) or in $p$-variation ($p\geq 1$) spaces. Such an equation is called \emph{Rough Differential Equation (RDE)} \cite{lyons98a,friz14a}.

This theory was very fruitful to study stochastic equations driven by Gaussian
processes which is not covered by the Itô framework, like the fractional Brownian
motion \cite{coutin02,unterberger10}. More generally,   rough path framework
allows one to separate the probabilistic from the deterministic part in such
equation and to overcome some probabilistic conditions such as using adapted
or non-anticipative processes.

Recently,  the ideas of rough path theory were extended to stochastic
partial differential equations (SPDE) with the works of \cite{hairer14,gub15}
which have led to significant progress in the study of some SPDE. This theory
also found applications in machine learning and the recognition of the
Chinese ideograms \cite{lyons14,chevyrev}.

Since the seminal article \cite{lyons98a} by T. Lyons in 1998, several
approaches emerged to solve~\eqref{eq:intro_rde}. They are
based on two main technical arguments: fixed point theorems
\cite{lyons98a,gub04} and flow approximations
\cite{feyel,coutin-lejay1,davie05a,bailleul12a,friz}. 
In particular, the rough flow theory allows one to extend
work about stochastic flows, which has been developed in 
'80s by Le Jan-Watanabe-Kunita and others,
to a non-semimartinagle setting~\cite{bailleul15a}.

The main goal of this
article is to give a framework which unifies the approaches by flow
and pursue further investigations on their properties and their 
relations with families of solutions to \eqref{eq:intro_rde}.

A \emph{flow} is a family of maps $\{\psi_{t,s}\}$ from a Banach space to
itself such that $\psi_{t,s}\circ\psi_{s,r}=\psi_{t,r}$ for any $r\leq s\leq
t$.  Typically, the map which associates the initial condition $a$ to the
solution of \eqref{eq:intro_rde} is expected to have a flow property.  The existence of a such
flow heavily depends on the existence and uniqueness of the solution. However,
it was proved in~\cite{cardona_semiflow,cardona_semiflow2} and extended to the
rough path case in \cite{brault1} that when non-uniqueness holds, it is
possible to build a measurable flow by a selection technique.  In this article
we are interested by the construction of a Lipschitz flows.

The main idea to build the flow associated is to find a good
approximation $\phi_{t,s}$ of $\psi_{t,s}$ when $\abs{t-s}$ is small enough.
We iterate this approximation on a subdivision $\pi=\{s\leq t_i\leq\dots\leq t_j\leq t\}$ 
of $[s,t]$ by setting
\begin{align*}
\phi_{t,s}^\pi\eqdef \phi_{t,t_j}\circ\dots\circ\phi_{t_i,s}.
\end{align*}
If $\phi^\pi$  converges when the mesh of $\pi$ goes to zero, $\phi^\pi$,
the limit is necessarily a flow.

This computation is similar to the ones of numerical schemes as
Euler's methods of different order \cite{chorin}. Moreover, this idea is found among
the Trotter's formulas for bounded or unbounded linear operators which allows to compute the semi-group of the sum of two non-commutative operators only knowing the semi-groups associated to each operator \cite{engel}. This property can be used to prove the Feynman-Kac formula.

Rather than working with a particular choice for the almost flow $\phi$
as in \cite{davie05a,friz}, we give here generic conditions on $\phi$.
We generalize the multiplicative sewing lemma of \cite{feyel} and
of \cite{coutin-lejay1}, introduced to solve linear RDE to a non linear situation. 
In this way, we construct directly some flows. 
In opposite to the additive and multiplicative sewing lemma, the limit
is not necessarily unique. The approximations are assumed to be Lipschitz. 
This is not the case for the limit. 

Our framework is close to the one developed by I.~Bailleul in \cite{bailleul12a,bailleul13b}
  as we also give two conditions which ensure that the iterated composition 
    of the almost flows remain uniformly Lipschitz continuous, which we called in \cite{brault1}
    the \emph{UL condition}. One condition
    corresponds to the \emph{$\cC^1$-approximate flow property} of \cite{bailleul12a}. 
    The other ones differs from the \emph{Regularity property} of \cite{bailleul12a}.

This condition, called the \emph{4-points control}, ensure that the almost 
flow satisfies the UL Condition. In spirit, it aggregates both the spatial and the temporal regularity 
into a single condition. We then study various consequences of this condition: existence
of an inverse, unique family of solutions, convergence of the Euler scheme,~...
The \emph{4-points control} can be checked on the almost flow, which is then called
a \emph{stable almost flow}. 
In \cite{brault3}, we show using 
Stochastic Differential Equations that a Lipschitz flow may exist
while the UL condition is not satisfied. We also exhibit a condition that ensure the uniqueness of the flow 
but which is weaker than the UL one.

We also study the relationship between almost flows and 
family of solutions to~\eqref{eq:intro_rde}
in the sense of Davie \cite{davie05a} as they are two different objects.   
In particular, we show that when an almost flow is stable, 
then the family of solutions to the RDE is unique and Lipschitz continuous. 
We also relate the distance between two families of solutions 
with respect to the distance between two almost flows when one
is stable. Again, consequences of this result will be drawn in~\cite{brault3}
where we study consistency and stability of the almost flows seen as approximations.
We also deduce from the results given here 
    that solutions to RDE are generic with respect to the vector field.  
    In \cite{lejay20}, we still use these results to consider differential 
    equations driven by rough paths living on general algebraic structures.

We also give several conditions under which perturbations of almost flows, 
a convenient tool to construct numerical schemes, converge to the same
limit flow. These perturbative arguments are the key to unify expansions
that are \textit{a priori} of different nature.

Finally, we apply our framework to recover the results of 
A.M.~Davie \cite{davie05a}, P.~Friz~\& N.~Victoir \cite{friz,friz2008}
and I.~Bailleul~\cite{bailleul12a,bailleul13b} using various 
perturbation arguments. As shown in \cite{lejay20}, our framework could be applied
to deal with branched rough paths, that are high-order expansions
indiced by trees, which are studied in~\cite{cass} and shown to 
fit Bailleul's framework~\cite{bailleul18}.

\noindent\textbf{Outline.} After introducing in Section~\ref{sec:notations} the
main notations and general definitions, we recall in
Section~\ref{sec:almost_flow} the notion of almost flow which is introduced in
our previous article \cite{brault1}. In Section~\ref{sec:stable_almost_flow},
we define the $4$-point control as well as stable almost flow~$\phi$.
We prove that under these conditions, $\phi^\pi$ converges to a Lipschitz flow.  
In Section~\ref{sec:perturbation_lip}, we give conditions to modify the
almost flow $\phi$ by adding a perturbation~$\epsilon$ which retains the
convergence to a flow.  We prove that under suitable conditions,  the
inverse of the approximation $\phi$  is a good approximation
of the inverse of the flow.  The link between uniqueness of the solution of
\eqref{eq:intro_rde} and  existence of a flow is studied in
Section~\ref{sec:generalized_solution_rde}.  In Section~\ref{sec:applications},
our formalism links the former approaches based on flow \cite{davie05a, friz,
bailleul12a}.  

\section{Notations}
\label{sec:notations}

The following notations and hypotheses will be constantly used throughout all this article.

\subsection{Controls and remainders}
\label{sec:controls}
Let us fix $T>0$, a time horizon. We write $\TT\eqdef [0,T]$ as well as 
\begin{align*}
    \rTT^2&\eqdef \Set{(s,t)\in\TT^2\given s\leq t}
    \text{ and }
    \rTT^3\eqdef \Set{(r,s,t)\in\TT^3\given r\leq s\leq t},\\
    \lTT^2:&=\Set{(s,t)\in\TT^2\given s\geq t}
    \text{ and }
    \lTT^3\eqdef \Set{(r,s,t)\in\TT^3\given r\geq s\geq t}.
\end{align*}
We also set $\TT^3=\rTT^3\cup\lTT^3$.

A \emph{control} $\omega$ is a family from 
$\rTT^2\eqdef\Set{0\leq s\leq t\leq T}$ to $\RR_+$ which is \emph{super-additive},
that is 
\begin{equation*}
    \omega_{r,s}+\omega_{s,t}\leq \omega_{r,t},\ \forall (r,s,t)\in \rTT^3,
\end{equation*}
and continuous close to its diagonal with $\omega_{s,s}=0$, $s\in\TT$.
For example $\omega_{s,t}=C|t-s|$ where $C$ is a non-negative constant.

A \emph{remainder} associated to a control $\omega$ is a continuous, increasing function $\varpi:{[0, \omega_{0,T})}\to \RR_+$ such that 
for some $0<\varkappa<1$, 
\begin{equation}
    \label{eq:h4}
    2\varpi\left(\frac{\tau}{2}\right)\leq \varkappa\varpi(\tau),\ \tau>0.
\end{equation}
A typical example for $\varpi$ is $\varpi(\tau)=\tau^\theta$ for any $\theta>1$. 
More generally, if $f$ is a continuous function such that $f(\tau/2)\leq \varkappa f(\tau)$
for any $\tau\geq $ with $\varkappa<1$, then $\varpi:\tau\mapsto \tau f(\tau)$ defines a remainder.
Such a function $f$ can be constructed from $s$-convex functions for example \cite{hudzik94a}.

Let $\delta:\RR_+\to \RR_+$ be non-decreasing function with $\lim_{T\to 0} \delta_T=0$.

We fix $\gamma\in \itvoc{0,1}$. 
We also consider a continuous, increasing function $\eta: {[0, \omega_{0,T})}\to \RR_+$ such that

\begin{equation}
\label{eq:h_eta}
\eta(\omega_{s,t})\varpi(\omega_{s,t})^\gamma\leq \delta_T\varpi(\omega_{s,t}),
\ \forall (s,t)\in\rTT^2.
\end{equation}

\subsection{Function spaces}

\label{sec:partitions}

We denote by $(\uV,\abs{\cdot})$ a Banach space.
The space of continuous functions from $\uV$ to $\uV$ 
is denoted by $\cC(\uV,{\uV})$. We set $\normsup{x}\eqdef \sup_{t\in[0,T]}\abs{x_t}$.

\begin{notation} 
    We denote by $\cF^+(\uV)$ the space of families $\Set{\phi_{t,s}}_{(s,t)\in\rTT^2}$ with $\phi_{t,s}\in\cC(\uV, {\uV})$ for each $(s,t)\in\rTT^2$.
    We also set $\cF^-(\uV)$ the space of families $\Set{\phi_{s,t}}_{(s,t)\in\rTT^2}$ with $\phi_{s,t}\in\cC(\uV, {\uV})$ for each $(s,t)\in\rTT^2$ (note the reversion of the indices).
\end{notation}

We now consider a partition $\pi=\Set{t_0\leq \dots \leq t_n}$ of $[0,T]$ with  a mesh denoted by~$\abs{\pi}$. 

\begin{notation}[Iterated products]
    \label{not:iteratedproducts}
For $\phi\in\cF^+(\uV)$, we write 
\begin{equation*}
    \phi_{t,s}^\pi\eqdef \phi_{t,t_j}\circ \phi_{t_j,t_{j-1}}
    \circ\dotsb\circ \phi_{t_{i+1},t_i}\circ \phi_{t_i,s},
\end{equation*}
where $[t_i,t_j]$ is the biggest interval of such kind contained in $[s,t]$.
We say that $\phi_{t,s}^\pi$ is the \emph{iterated product} of $\phi$ on a subdivision $\pi$.
If no such interval exists, then $\phi^\pi_{t,s}=\phi_{t,s}$. 
\end{notation}

For $\phi\in\cF^+(\uV)$, we define similarly 
\begin{equation*}
\phi_{s,t}^\pi\eqdef \phi_{s,t_1}\circ \phi_{t_1,t_2} \circ\dotsb\circ \phi_{t_{j-1},t_j}\circ \phi_{t_j,t}.
\end{equation*}

For any partition $\pi$, $\phi^\pi\in\cF^\pm(\uV)$ when $\phi\in\cF^\pm(\uV)$.
A trivial but important remark is that 
from the very construction, 
\begin{equation*}
    \phi_{t,s}^\pi=\phi_{t,r}^\pi\circ \phi^\pi_{r,s}\text{ for any }r\in\pi.
\end{equation*}
In particular, $\Set{\phi^\pi_{t,s}}_{(s,t)\in\rlTT^2,\ s,t\in\pi}$ enjoys a (semi-)flow 
property (Definition~\ref{def:flow}). A natural question is then to study the limit of $\phi^\pi$ as the mesh of $\pi$
decreases to $0$.

Finally, for any $(r,s,t)\in\rlTT^3$ we write $\phi_{t,s,r}\eqdef \phi_{t,s}\circ\phi_{s,r}-\phi_{t,r}$.

\begin{notation}
    \label{not:4}
    We extend the norm $\abs{\cdot}$ on $\cF^\pm(\uV)$ by 
    \begin{equation*}
	\normf{\varpi}{\phi}\eqdef\sup_{\substack{(s,t)\in\rlTT^2\\ s\not=t}}
	    \frac{\normsup{\phi_{t,s}}}{\varpi(\omega_{s,t})},
    \end{equation*}
where $\omega$, $\varpi$ are defined in Section~\ref{sec:controls}. 
Possibly, $\normf{\varpi}{\phi}=\infty$. Actually, this norm 
is mainly used to consider the distance between two elements
of $\cF^\pm(\uV)$.
With this norm, $(\cF^\pm(\uV),\normf{\varpi}{\cdot})$ is a Banach space. 
\end{notation} 
\begin{definition} 
\label{def_galaxy}
    We define the equivalence relation $\sim$ on $\cF^\pm(\uV)$ by 
    $\phi\sim\psi$ if and only if there exists a constant $C$ such that 
    \begin{equation*}
  \normsup{\phi_{t,s}-\psi_{t,s}}\leq C\varpi(\omega_{s,t}),
    \  \forall (s,t)\in\TT^2.    
    \end{equation*}
    In other words, $\phi\sim\psi$ if and only if $\normf{\varpi}{\phi-\psi}<+\infty$. 
    Each quotient class of $\cF^\pm(\uV)/\sim$ is called a \emph{galaxy}, which 
    contains elements of $\cF^\pm(\uV)$ which are at finite distance from each others.
\end{definition}

\begin{notation}[Lipschitz semi-norm]
    The Lipschitz semi-norm of a function $f$ from a Banach space 
    $(\uV,|\cdot|)$ to another Banach space $(\uW,|\cdot|')$ is 
    \begin{equation*}
	\normlip{f}\eqdef\sup_{\substack{a,b\in\uV,\\a\neq b}} \frac{|f(a)-f(b)|'}{|a-b|},
    \end{equation*}
    whenever this quantity is finite. And if $A\subset \uV$ is a
    non-empty subset of $\uV$, we say that $f$ is  Lipschitz
    continuous on $A$ when
    \begin{align*}
    \normlipl{f}{A}\eqdef\sup_{\substack{a,b\in A,\\a\neq b}} \frac{|f(a)-f(b)|'}{|a-b|}<+\infty.
    \end{align*}
\end{notation}

\begin{notation}[Hölder spaces]
For $\gamma\in (0,1)$,  an integer $r$ and two Banach spaces $\uV_1$, $\uV_2$, we denote by $\cC^{r+\gamma}_b({\uV_1,\uV_2})$ the space of bounded continuous functions from $\uV_1$ to $\uV_2$ with bounded (Fréchet) derivatives up to order $r$ and a $r$ order derivative which is $\gamma$-Hölder continuous. 
We also denote by $L(\uV_1,\uV_2)$ the set of continuous linear maps from $\uV_1$ to $\uV_2$.
\end{notation}

\section{Almost flow and Uniform Lipschitz condition}
\label{sec:almost_flow}
In this section, we recall some notions and results introduced in \cite{brault1},
which are useful in next sections. As we are working on Banach spaces
instead of metric spaces, 
we have a slightly stronger notion of almost flow than in \cite{brault1}. 

We denote by $\id$ the identity map from $\uV$ to $\uV$.

\begin{definition}[Almost flow]
\label{def:almost_flow}
An element $\phi\in\cF^+(\uV)$ is an \emph{almost flow} if for any $T>0$ and 
any $(r,s,t)\in\rTT^3$, $a,b\in\uV$,
\begin{gather}
    \label{eq:h0}
    \phi_{t,t}=\id,\\
    \label{eq:h1}
    \normsup{\phi_{t,s}-\id}\leq \delta_T,\\
    \label{eq:h2}
    \abs{\phi_{t,s}(b)-\phi_{t,s}(a)}\leq (1+\delta_T)\abs{b-a}+\eta(\omega_{s,t})\abs{b-a}^\gamma,\\
    \label{eq:h3}
    \normsup{\phi_{t,s,r}}\leq M\varpi(\omega_{r,t}),
\end{gather}
where $M\geq 0$ and $\phi_{t,s,r}\eqdef \phi_{t,s}\circ\phi_{s,r}-\phi_{t,r}$. If we replace $(r,s,t)\in\rTT^3$ by $(r,s,t)\in\lTT^3$, we say that $\phi$ is a \emph{reverse almost flow}.
\end{definition}
\begin{definition}[Semi-flow and Flow] 
\label{def:flow}
A \emph{semi-flow} $\psi$ is a family of functions $(\psi_{t,s})_{(s,t)\in\rTT^2}$ from $\uV$ to $\uV$ such that
    $\psi_{t,t}=\id$ and 
\begin{align}
\label{eq:def_flow}
\psi_{t,s}\circ\psi_{s,r}=\psi_{t,r}
\end{align}    
     for any $a\in\uV$ and $(r,s,t)\in\rTT^3$. It is a \emph{flow} if
     each $\psi_{t,s}$ is invertible with an inverse equal to $\psi_{s,t}$ 
     for any $(s,t)\in\rTT^2$ and 
     \eqref{eq:def_flow} holds for any $(r,s,t)\in\TT^3\eqdef \rTT^3\cup\lTT^3$.
\end{definition}

\begin{remark}
    The inverse of $\psi_{t,s}$ is $\psi_{s,t}$ for $(s,t)\in\TT^2$.
\end{remark}
\begin{theorem}[{\cite{brault1}}]
    \label{thm:1}
    Let $\phi$ be an almost flow (Definition~\ref{def:almost_flow})
    with $M\geq 0$ and $\delta_T$, $\kappa$ defined in
    Section~\ref{sec:controls} Then there exists a time horizon $T$
    small enough and a constant
    $L\leq 2M/(1-(1+\delta_T)\kappa-\delta_T)$ such that
    \begin{equation}
	\label{eq:sewing_sup}
	\normsup{\phi^\pi_{t,s}-\phi_{t,s}}\leq L\varpi(\omega_{s,t})
    \end{equation}
    for any $(s,t)\in\rTT^2$ and any  partition $\pi$ of $\TT$.
\end{theorem}
\begin{definition}[Condition UL]
\label{prop:ul_condition}
An almost flow $\phi$ such that $\normlip{\phi^\pi_{s,t}}\leq 1+\delta_T$
for any $(s,t)\in\rTT^2$ whatever the partition $\pi$ is said to satisfy 
the \emph{uniform Lipschitz} (UL) condition.
\end{definition}

We give a sufficient condition on an almost flow to get a Lipschitz flow in a galaxy.

\begin{proposition}
    Let $\phi$ be an almost flow which satisfies the condition UL.
 Then there exists a Lipschitz flow $\psi$ with $\normlip{\psi_{t,s}}\leq 1+\delta_T$ 
 for any $(s,t)\in\rTT^2$ 
 such that $\phi^\pi_{t,s}(a)$ converges to $\psi_{t,s}(a)$
 for any $a\in\uV$ and any $(s,t)\in\rTT^2$.
\end{proposition}

On the other hand, there could be at most one flow in a galaxy
if one is Lipschitz.

\begin{proposition}
\label{prop:unique_flow}
Assume that there is a Lipschitz flow $\psi$ in a galaxy $G$.
Then $\psi$ is the unique flow in~$G$.
Moreover, for any almost flow $\phi\sim\psi$, $\phi^\pi_{s,t}(a)$
converges to $\psi_{s,t}(a)$ for any $(s,t)\in\rTT^2$ and $a\in\uV$.
\end{proposition}

We then complete the results of \cite{brault1} with 
the following ones. 

\begin{proposition}
    Let $\phi$ be an almost flow which satisfies the condition UL.
    Then $\phi^\pi$ is an almost flow for any partition $\pi$.
\end{proposition}

The prototypical example for the next result is $\varpi(\tau)=\tau^\theta$
for some $\theta>1$. In this case, it slightly improves the 
rate of convergence as $\theta-1$ with respect to the one given in \cite{brault1}
which is $\theta-1-\epsilon$ for any $\epsilon>0$.

\begin{proposition}[Rate of convergence]
    \label{prop:rate}
    Let $\phi$ be an almost flow in the same galaxy as a Lipschitz flow 
    $\psi$ with $\normlip{\psi_{t,s}}\leq (1+\delta_T)$  
    and $\normsup{\psi_{t,s}-\phi_{t,s}}\leq K\varpi(\omega_{s,t})$
for any $(s,t)\in\rTT^2$.
Let us assume that $\varpi$ is such that 
for a bounded function $\mu$, 
$\tau^{-1}\varpi(\tau)\leq \mu(\tau)$ for any $\tau>0$.
Then 
\begin{equation*}
    \normsup{\psi_{t,s}(a)-\phi^\pi_{t,s}(a)}\leq KM(\pi)\omega_{0,T}(1+\delta_T)
    \text{ with }
    M(\pi)\eqdef \sup_{\substack{(r,t)\text{ successive}\\\text{points in }\pi}}\mu(\omega_{r,t}).
\end{equation*}
\end{proposition}
\begin{proof}
    The proof follows the one of Theorem~10.30 in \cite[p.~238]{friz}.
    Let $\Set{t_i}_{i=0,\dotsc,n}$ be the points of $\pi\cup\Set{s,t}$ such that 
    $t_0=s$, $t_n=t$ and $(t_i,t_{i+1})$ are successive points in $\pi\cup\Set{s,t}$.
    Set $z_k=\psi_{t,t_k}(\phi_{t_k,t}(a))$. Then 
    \begin{equation*}
	\abs{\psi_{t,s}(a)-\phi^\pi_{t,s}(a)}
	=
	\abs{z_n-z_0}
	\leq
	\sum_{i=0}^{n-1}\abs{z_{i+1}-z_i}. 
    \end{equation*}
    Since $\psi_{t,s}$ is Lipschitz and
    \begin{equation*}
	\abs{z_{i+1}-z_i}\leq (1+\delta_T)\abs{\psi_{t_{k+1},t_k}(\phi_{t_{k+1},t_k}(a))
	-\psi_{t_{k+1},t_k}(\phi_{t_{k+1},t_k}(a))}.
    \end{equation*}
    The result follows easily.
\end{proof}

\section{Stable almost flows}
\label{sec:stable_almost_flow}

In the previous section, we have recalled some results 
from \cite{brault1} which underline the importance of the uniform Lipschitz (UL) condition. 
However, the UL condition is not easy to verify as one has to control uniformly iterated products
of almost flows. A natural approach is then to focus on conditions on the almost flow alone.
In this section, we give a sufficient condition on an almost flow $\phi$ to ensure
that it satisfies the UL condition and then that its galaxy contains
a unique flow which is Lipschitz.

\subsection{The 4-points control}
In this section $\uV, \uV_1,\uV_2,\uV_3$ are Banach spaces and we denote by $\eucnorm{\cdot}$ their norms.
 \begin{definition}[The $4$-points control]
    A function $f:\uV_1\to\uV_2$ is said to satisfy 
    a \emph{$4$-points control} if there exists a non-decreasing, continuous function
    $\widehat{f}:\RR_+\to\RR_+$ and a constant $\wcf\geq 0$ such 
    that 
    \begin{multline}
	\label{eq:def_4_points}
	\abs{f(a)-f(b)-f(c)+f(d)}
	\\
	\leq
	\widehat{f}(\abs{a-b}\vee\abs{c-d})
	\times \left(\abs{a-c}\vee\abs{b-d}\right)+\wcf\abs{a-b-c+d} 
    \end{multline}
    for any $(a,b,c,d)\in\uV_1$.
\end{definition}

Any  Lipschitz function $f$ satisfies
a $4$-points control with $\widehat{f}=2\normlip{f}$ and $\wcf=0$. 
However, for our purpose, we need to consider later more restrictive conditions
on~$\wcf$ and $\widehat{f}$.

Let us start with a simple example. When the derivative of $f$ is $\gamma$-Hölder continuous, 
it restates \cite[Leamm 3.5]{davie05a}.

\begin{lemma}
\label{lem:4_points_f}
Let $f\in\cC^{1}(\uV_1,\uV_2)$, $0<\gamma\leq 1$, with a bounded derivative which
is uniformly continuous with a modulus of continuity $\nu(\tau):=\sup_{\abs{x-y}\leq \tau}\abs{{\nabla}f(y)-{\nabla}f(x)}$.
 Then $f$ satisfies a $4$-points control with 
    \begin{equation*}
	\widehat{f}(x):=\nu(x),\ x\geq 0, 
	\text{ and }\wcf:=\normsup{\nabla f}.
    \end{equation*}
\end{lemma}
\begin{proof}
For any $a,b,c,d\in \uV_1$,
\begin{multline*}
f(a)-f(b)-f(c)+f(d)
\\
=(a-b)\int_0^1\nabla f(au+(1-u)b)\vd u
-(c-d)\int_0^1\nabla f(cu+(1-u)d)\vd u
\\
=(a-b)\int_0^1[\nabla f(au+(1-u)b)-\nabla f(cu+(1-u)d)]\vd u
\\
+(a-b-c+d)\int_0^1\nabla f(cu+(1-u)d)\vd u,
\end{multline*}
which yields to the result.
\end{proof}

Here are a few properties of functions satisfying a $4$-points control. 

\begin{lemma}
    \label{lem:4pc:1}
Let $f,g$ satisfying a $4$-points control with $g$ Lipschitz continuous.
\begin{enumerate}[thm]
    \item\label{lem:4pc:1:i} The function $f$ is locally Lipschitz continuous.
    \item\label{lem:4pc:1:ii} If $f,g:\uV_1\to\uV_2$, then for any $\lambda,\mu\in\RR$, $\lambda f+\mu g$ satisfies a $4$-points control.
    \item\label{lem:4pc:1:iii} If $f:\uV_1\to\uV_2$ and $g:\uV_3\to\uV_1$, then $f\circ g:\uV_3\to\uV_2$ satisfies a $4$-points control. 
    \item\label{lem:4pc:1:iv} If $g:\uV\to\uV$ with $\normlip{g}<1$, then $\id+g$ is invertible and $k\eqdef (\id+g)^{-1}$ is Lipschitz and 
	satisfies a $4$-points control with 
	$\widehat{k}(x)=\widehat{g}(\normlip{k}x)\normlip{k}$
	for $\normlip{k}\leq 1/(1-\normlip{g})$, 
	and $\oldwidecheck{k}=1/(1-\oldwidecheck{g})$.
\end{enumerate}
\end{lemma}

\begin{proof}
For showing \ref{lem:4pc:1:i}, we choose $a=d$ and $b=c$ in \eqref{eq:def_4_points} and then,
\begin{equation*}
|f(a)-f(b)|\leq \left[\widehat{f}(|a-b|)+\wcf\right]|a-b|,
\end{equation*}
which proves that $f$ is locally Lipschitz continuous.

Moreover, we can choose $\widehat{\lambda f+\mu g}=|\lambda|\widehat{f}+|\mu|\widehat{g}$ and $\oldwidecheck{(\lambda f+\mu g)}=|\lambda|\wcf+|\mu|\oldwidecheck{g}$ to obtain a $4$-points control on $\lambda f+\mu g$. This proves \ref{lem:4pc:1:ii}.

To show \ref{lem:4pc:1:iii}, we use the fact that $g$ is Lipschitz according to \ref{lem:4pc:1:i}. 
With $h=f\circ g$, 
\begin{multline*}
|h(a)-h(b)-h(c)+h(d)|
\\
\leq \widehat{f}(|g(a)-g(b)|\vee |g(c)-g(d)|)\left[|g(a)-g(c)|\vee |g(b)-g(d)|\right]
\\
 +\wcf\widehat{g}(|a-b|\vee |c-d|)\left[|a-c|\vee |b-d|\right]+\wcf\oldwidecheck{g}|a-b-c+d|
\\
\leq \widehat{f}(\normlip{g}|a-b|\vee |c-d|)\normlip{g}(|a-c|\vee |b-d|)\\
+\wcf\widehat{g}(|a-b|\vee |c-d|)[|a-c|\vee |b-d|]+\wcf\oldwidecheck{g}|a-b-c+d|.
\end{multline*}
This proves that $h$ satisfies the $4$-points control.

It remains to show \ref{lem:4pc:1:iv}. 
From the Lipschitz inverse function theorem \cite[p.~124]{abraham}, $\id+g$ is invertible
with an inverse $k$ that satisfies $\normlip{k}\leq (1-\normlip{g})^{-1}$. 
Besides, 
\begin{align*}
    |a-b-c+d+(g(a)-g(b)-g(c)+g(d))|&\geq (1-\oldwidecheck{g})|a-b-c+d|\\
				   &-\widehat{g}(|a-b|\vee |c-d|)\times\left[|a-c|\vee |b-d|\right],
\end{align*}
which yields to
\begin{align*}
    (1-\oldwidecheck{g})|k(a)-k(b)-k(c)+k(d)|&\leq |a-b-c+d|\\
					  &+\widehat{g}(\normlip{k}|a-b|\vee |c-d|)\normlip{k}\left[|a-c|\vee |b-d|\right].
\end{align*}
Hence, for $x\in\RR_+$, $\widehat{k}= \widehat{g}(\normlip{k} x)\normlip{k}$ and $\oldwidecheck{k}=(1-\oldwidecheck{g})^{-1}$. Therefore, $k$ satisfies a 4-points control.
\end{proof}

The reason for introducing the 4-points control lies in its good behavior with 
respect to composition. More precisely, if $f$ satisfies a 4-points control
while $g$ and $h$ are Lipschitz continuous and bounded,
\begin{align*}
    \normlip{f\circ g-f\circ h}&
    \leq \widehat{f}(\normsup{g-h})\normlip{g}\vee \normlip{h}
    +\oldwidecheck{f}\normlip{g-h},
    \\
    \normsup{f\circ g-f\circ h}&
    \leq \left(\widehat{f}(0)+\oldwidecheck{f}\right)\normsup{g-h}.
\end{align*}


\subsection{Definition of a stable almost flow}

\begin{definition}[{$\varpi$-compatible 4-points control}]
  \label{def:varpi-compatible}
  A family $\phi\in\cF^+(V)$ is said to satisfy a
  \emph{$\varpi$-compatible $4$-points control} if there exists a
  family of functions $(\widehat{\phi}_{t,s})_{(s,t)\in \rTT^2}$ and
  constants $(\oldwidecheck{\phi}_{t,s})_{(s,t)\in\TT^2}$ such that for
  any $[s,t]\subset [0,T]$ the estimate \eqref{eq:def_4_points}
  holds and  
  \begin{align}
      \label{eq:varpi-compatible:1}
      \widehat{\phi}_{s,t}&\leq \widehat{\phi}_{u,v}\text{ for any } [s,t]\subset[u,v]\\
      \label{eq:varpi-compatible:2}
    \widehat{\phi}_{t,s}(\alpha\varpi(\omega_{s,t}))&\leq \phi^{\circledast}(\alpha)\varpi(\omega_{s,t})
    \text{ for any }\alpha\geq 0,\ (s,t)\in\rTT^2,
    \end{align}
    where $\phi^\circledast$ is a non-negative function from $\RR_+\to\RR_+$.
\end{definition}
\begin{definition} [Stable almost flow]
    \label{def:stable_almost_flow}
    We say that an almost flow $\phi$ with $\gamma=1$ in~\eqref{eq:h2} is a \emph{stable almost flow} if 
    \begin{enumerate}[thm]
	\item\label{point:saf:i} it satisfies a $\varpi$-compatible $4$-points control
          with
          \begin{align*}
            \oldwidecheck{\phi}_{t,s}\leq 1+\delta_T,
            \end{align*}
	\item\label{point:saf:ii} there is a constant $C\geq 0$ such that for $(r,s,t)\in \rTT^3$, 
	    \begin{equation}
		\label{eq:prop_lip_t_s_r}
		\normlip{\phi_{t,s,r}}\leq C \varpi(\omega_{r,t}),
	    \end{equation}
	    where $\phi_{t,s,r}\coleq \phi_{t,s}\circ\phi_{s,r}-\phi_{t,r}$.
\end{enumerate}
We denote the family of stable almost flow $\cS\cA_{\delta_T,\varpi}(V)$.
If we replace assumption $(r,s,t)\in\rTT^3$ by $(r,s,t)\in\lTT^3$, we say that $\phi$ is a \emph{reverse stable almost flow}.
\end{definition}

    \begin{remark} Condition~\eqref{eq:prop_lip_t_s_r} in \ref{point:saf:ii}
	is the Condition H2 (\textquote{$\cC^1$-approximate flow property}) of \cite[Theorem~2.1]{bailleul12a}. 
	The conditions in \ref{point:saf:i} is an alternative to Condition H1 (\textquote{regularity})
	of \cite[Theorem 2.1]{bailleul12a}. Theorem 2 belows  leads to similar conclusions in spirit
	to \cite[Theorem 2.1]{bailleul12a}. Other results have their counterparts in \cite{bailleul12a,bailleul18}.
\end{remark}

\begin{remark}
    A stable almost flow is necessarily Lipschitz, so that  \eqref{eq:h2} in  Definition~\ref{def:almost_flow} holds with
    $\gamma=1$ or $\eta=0$. Lipschitz flows are the central object of this study. However, 
    in \cite{brault3}, we weaken the Lipschitz regularity while ensuring the limiting flow is
    unique in its galaxy.
\end{remark}

    Let us end this section with a simple yet practical example to the Young differential equation, 
    that is approximation of differential equation driven by a path not too irregular so 
    that it implies a Young integral. This example does not focus on a particular kind of 
    path, but shows the interplay between the modulus of continuity of the driving path 
    and the regularity of the vector field. 
    In Section~\ref{sec:applications}, we consider various kind of approximations to deal with RDE. 

\begin{lemma}
Fix $m,d\geq 1$. We choose as control $\omega_{s,t}:=t-s$.
Let $x:[0,T]\to\RR^d$ be a continuous path with modulus of continuity $\mu$, that is 
$\abs{x_{s,t}}\leq \mu(t-s)$  where $x_{s,t}:=x_t-x_s$ for $(s,t)\in\rTT^2$
For $i=1,\dotsc,d$, we consider a family of functions $f_i\in\cC^{1}(\RR^m,\RR^m)$
We assume that each $\nabla f_i$ is uniformly continuous. 
We set 
\begin{equation*}
 \nu(\tau):=\max_{i=1,\dotsc,d}
\sup_{\abs{y-x}\leq \tau} \abs{\nabla f_i(y)-\nabla f_i(x)}   
\text{ and }\normsup{\nabla^k f}:=\max_{i=1,\dotsc,d} \normsup{\nabla^k f_i}
\end{equation*}
for $k=0,1$. We assume that for any $\alpha\geq 0$ small enough, there exists 
a constant $\nu^\circledast(\alpha)$ such that 
\begin{equation*}
    \sup_{0<\tau\leq \mu(T)} \frac{\nu(\alpha\tau)}{\nu(\tau)}\leq \nu^\circledast(\tau).
\end{equation*}
We define 
\begin{equation*}
    \phi_{t,s}(a):=a+f(a)x_{s,t}=a+\sum_{i=1}^d f_i(a)x^i_{s,t},\ (s,t)\in\rTT^2,\ a\in\uV
\end{equation*}
and 
\begin{equation*}
\varpi(\tau):=\max\Set{\nu(\mu(\tau)\big)\mu(\tau),\mu(\tau)^2}.
\end{equation*}
If $\varpi$ is a remainder (See Section~\ref{sec:notations}),
then for $T$ small enough, $\phi:=\Set{\phi_{t,s}}_{(s,t)\in\rTT^2}$ is a stable
almost flow.
\end{lemma}

\begin{proof}
    Clearly, $\phi$ satisfies \eqref{eq:h0}-\eqref{eq:h2}.
    With Lemma~\ref{lem:4_points_f}, for any $a,b,c,d\in\RR^m$ and any $(s,t)\in\rTT^2$, 
    \begin{multline*}
	\abs{\phi_{t,s}(a)-\phi_{t,s}(b)-\phi_{t,s}(c)+\phi_{s,t}(d)}
	\leq 
	\abs{a-b-c+d}\left(1+\normsup{\nabla f}\abs{x_{s,t}}\right)
	\\
	+\abs{x_{s,t}}
	\nu(\abs{a-c}\vee \abs{b-d})
	\cdot ( \abs{a-b}\vee \abs{c-d}).
    \end{multline*}
    Therefore, $\phi_{t,s}$ satisfies the 4-points control with 
    \begin{equation*}
	\widehat{\phi_{t,s}}(\rho):=\nu(\rho)\mu(t-s) \nu^\circledast(\normsup{f}\vee1)
      \text{ and }
\phi_{t,s}^\dag:=1+\normsup{\nabla f}\cdot\mu(T)
    \end{equation*}
    
From our construction of $\varpi$, for any $\alpha\geq 0$, 
\begin{equation*}
    \widehat{\phi_{t,s}}(\alpha\varpi(t-s))
    \leq \mu(t-s)\nu\big(\max\Set{\nu(\mu(t-s)),\mu^2(t-s)}\mu(t-s)\big)\nu^\circledast(\alpha)\nu^\circledast(\normsup{f}\vee 1)
\end{equation*}
Since $\nu$ is a modulus of continuity, $\nu(\mu(\tau))\leq \mu(\tau)$. 
Provided that $\tau$ is small enough so that $\mu(\tau)\leq 1$, 
then $\widehat{\phi_{t,s}}$
satisfies \eqref{eq:varpi-compatible:2}. It also clearly satisfies~\eqref{eq:varpi-compatible:1}.
Thus, $\phi$ satisfies a $\varpi$-compatible 4-points control.

Since for any $(r,s,t)\in\rTT^3$, $x_{r,s}+x_{s,t}=x_{r,t}$, a standard computation shows that 
    \begin{equation*}
	\phi_{t,s,r}(a)=f(a+f(a)x_{r,s})x_{s,t}-f(a)x_{s,t}
	\text{ for any }(r,s,t)\in\rTT^3.
    \end{equation*}
    With our choice of $\varpi$, 
    \begin{equation*}
	\normsup{\phi_{t,s,r}}\leq \normsup{\nabla f}\normsup{f}\mu(t-r)^2
	\leq \normsup{\nabla f}\normsup{f}\varpi(t-r)^2.
    \end{equation*}
    Hence, \eqref{eq:h3} is satisfies and $\phi$ is an almost flow.

    Again using Lemma~\ref{lem:4_points_f}, for any $a,b\in\uV$ and any $(r,s,t)\in\rTT^3$,
    \begin{multline*}
	\abs{\phi_{t,s,r}(b)-\phi_{t,s,r}(a)}
	\leq 
	\\
	\abs{a-b}
	\nu(\normsup{f}\abs{x_{s,t}})\abs{x_{s,t}}(1+\normsup{\nabla f}\abs{x_{r,s}})
	+\abs{a-b}\cdot\normsup{\nabla f}\cdot\abs{x_{r,s}}\cdot\abs{x_{s,t}}
	\\
	\leq \abs{a-b}\varpi(t-r)\nu^\circledast(\normsup{f})(1+\normsup{\nabla f}).
    \end{multline*}
    This proves that $\phi$ satisfies also \ref{point:saf:ii} and is a stable almost flow.
\end{proof}

\begin{example}
    \label{ex:1} The Young case corresponds to $\mu(\tau)=\tau^\alpha$ for some $\alpha\in(0,1]$
    and $\nu(\tau)=C\tau^\gamma$ for some $\gamma\in(0,1]$ with $\alpha(1+\gamma)>1$
    (and then $\alpha>1/2$). This choice of $\nu$ means that the derivative of $f$
    is $\gamma$-Hölder continuous.
    This way, 
    \begin{equation*}
	\varpi(\tau):=C\tau^{\alpha(1+\gamma)}\vee \tau^{2\alpha}.
    \end{equation*}
    which satisfies \eqref{eq:h4} as $\alpha(1+\gamma)>1$ and $2\alpha>1$.
\end{example}


\subsection{Non linear sewing lemma for stable almost flow}

In this section, we introduce our main tools to control the iterated products
(Notation~\ref{not:iteratedproducts}) on a partition.  The key idea
is borrowed from the claim in \cite[p.~6]{davie05a}.

\begin{definition}[Successive points]
    \label{def:successivepoints}
Let $\pi$ be a partition of $[0,T]$. Two points $s$ and $t$ of $\pi$
are said to be \emph{at distance $k$} if there are $k-1$ points between them in $\pi$. 
Points at distance $1$ are then \emph{successive points} in $\pi$. 
\end{definition}

\begin{lemma} 
    \label{lem:davie}
    Let us consider a family $U\eqdef \Set{U_{s,t}}_{s,t\in\pi,s\leq t}$ with values in $\RR_+$ satisfying
     for any $(r,s,t)\in\pi\bigcap \TT^3,$ 
    \begin{gather}
    \notag
    U_{r,s}\leq D\varpi(\omega_{r,s})\text{ when $r$ and $s$ are successive points},\\
    \label{eq:n2_new}
    U_{r,t}\leq(1+\alpha_T)U_{r,s}+(1+\alpha_T)U_{s,t}+B\varpi(\omega_{r,t}), 
    \end{gather}
    for some constants $D\geq 1$, $B\geq 0$ and $\alpha_T\geq 0$ that decreases to $0$ as $T\to 0$.

    Then for all $T>0$ such that $\varkappa(1+\alpha_T)^2+\delta_T<1$, 
    \begin{gather}
	\label{eq:davie:4}
	U_{r,t}\leq A\varpi(\omega_{r,t}),\ \forall (r,t)\in[0,T]\cap\pi^2,\\
	\label{eq:davie:4:cst}
	\text{ with }A\eqdef \frac{D(1+\delta_T)(1+\alpha_T)^2+B(2+\alpha_T)}{1-(\varkappa(1+\alpha_T)^2+\delta_T)}.
    \end{gather}
    In particular, $A$ does not depend on the choice of the partition.
\end{lemma}
\begin{proof}
We perform an induction on the distance $m$ between points in $\pi$.

If $m=1$, then \eqref{eq:davie:4} is true since that $A\geq D$. 

Let us assume that this is true for any two points at distance $m$. Fix 
two points $r$ and $t$ at distance $m+1$ in $\pi$. Hence, there exists
two successive points $s$ and $s'$ in $\pi$ such that 
\begin{equation*}
    \omega_{r,s}\leq \frac{\omega_{r,t}}{2}
    \text{ and }\omega_{s',t}\leq \frac{\omega_{r,t}}{2}.
\end{equation*}

Applying \eqref{eq:n2_new} twice with $(r,s,t)$ and $(s,s',t)$,
\begin{multline*}
    U_{r,t}\leq (1+\alpha_T) U_{r,s}+(1+\alpha_T) U_{s,t}+B\varpi(\omega_{r,t})\\
    \leq (1+\alpha_T)U_{r,s}+(1+\alpha_T)^2(U_{s,s'}+U_{s',t})
    +(2+\alpha_T)B\varpi(\omega_{r,t}).
\end{multline*}
Both $U_{r,s}$ and $U_{s,t}$ satisfy the induction property.
With \eqref{eq:h4}, 
\begin{multline*}
U_{r,t} \leq 
2A(1+\alpha_T)^2\varpi\left(\frac{\omega_{r,t}}{2}\right)
    +(1+\alpha_T)^2D\varpi(\omega_{s,s'})
    +(2+\alpha_T)B\varpi(\omega_{r,t})\\
    \leq \left[A\left(\varkappa(1+\alpha_T)^2+\delta_T\right)
    +D(1+\alpha_T)^2(1+\delta_T)+(2+\alpha_T)B\right]\varpi(\omega_{r,t}).
\end{multline*}
Our choice of $A\geq D$ in \eqref{eq:davie:4:cst} ensures the results at level $m+1$.
The control \eqref{eq:davie:4} is then true whatever the partition.
\end{proof}

The following proposition justifies Definition~\ref{def:stable_almost_flow}.

\begin{theorem}
\label{prop:stable_almost_flow_UL}
If $\phi\in\cS\cA_{\delta_T,\varpi}(\uV)$ is a stable almost flow 
then for any partition $\pi$, 
\begin{align}
\label{eq:prop_UL_4_points}
\normlip{\phi^\pi_{t,s}-\phi_{t,s}}\leq L\varpi(\omega_{s,t}), \forall (s,t)\in\rTT^2,
\end{align}
where $L$ is a constant that depends on $T$, $T\mapsto \delta_T$, $\varkappa$, $\omega$,
$\widehat{\phi}$, $\phi^\circledast$ and $C$ in \eqref{eq:prop_lip_t_s_r}. In particular, 
the almost flow $\phi$ satisfies the condition UL, up to changing $\delta_T$.
\end{theorem}

\begin{remark}
 When $\phi$ is a stable almost flow, we assume that $\gamma=1$ in
 Definition~\ref{def:almost_flow}. This implies that \eqref{eq:h2} becomes
  \begin{align}
        \label{eq:h2:gamma1}
    \abs{\phi_{t,s}(b)-\phi_{t,s}(a)}\leq (1+\delta_T)\abs{b-a}.
  \end{align}
\end{remark}

\begin{proof}
Let us choose a partition $\pi$. 
Let $r\in\TT$ be fixed and  $(s,t)\in\rTT^2,$ such that $r\leq s$,
\begin{equation*}
    U^\pi_{s,t}\eqdef \normlip{\phi_{t,r}^\pi-\phi_{t,s}\circ\phi_{s,r}^\pi}.
\end{equation*}

For any $0\leq r\leq s\leq t\leq u\leq T$, 
\begin{align}
    U^\pi_{s,u}\leq U^\pi_{t,u}+\normlip{\phi_{u,t}\circ\phi_{t,r}^\pi-\phi_{u,t}\circ\phi_{t,s}
\circ\phi^\pi_{s,r}}+\normlip{\phi_{u,t}\circ\phi_{t,s}\circ\phi_{s,r}^\pi-\phi_{u,s}\circ\phi_{s,r}^\pi}.
\label{eq:proof_decomp}
\end{align}
With the $4$-points control on $\phi_{u,t}$, 
\begin{multline*}
\normlip{\phi_{u,t}\circ\phi_{t,r}^\pi-\phi_{u,t}\circ\phi_{t,s}
\circ\phi^\pi_{s,r}}\\
\leq \widehat{\phi}_{u,t}\left(\normsup{\phi^\pi_{t,r}-\phi_{t,s}\circ\phi_{s,r}^\pi}\right)
\times\left(\normlip{\phi_{t,r}^\pi}\vee(1+\delta_T)
\normlip{\phi_{s,r}^\pi}\right)
+\oldwidecheck{\phi}_{u,t}U^\pi_{s,t}.
\end{multline*}
According to \eqref{eq:sewing_sup}  for $T$ small enough,
\begin{equation*}
\normsup{\phi^\pi_{t,r}-\phi_{t,s}\circ\phi_{s,r}^\pi}\leq C\varpi(\omega_{s,t}).
\end{equation*}
Since $\widehat{\phi_{u,t}}$ satisfies \eqref{eq:varpi-compatible:1}-\eqref{eq:varpi-compatible:2}, 
with the control~\eqref{eq:h2} with $\gamma=1$ of the Definition~\ref{def:almost_flow}
of almost flow, 
\begin{multline*}
\normlip{\phi_{u,t}\circ\phi_{t,r}^\pi-\phi_{u,t}\circ\phi_{t,s}
\circ\phi^\pi_{s,r}}\\
\leq \phi^\circledast(C)\varpi(\omega_{r,t})\left(\normlip{\phi_{t,r}^\pi}\vee(1+\delta_T)
\normlip{\phi_{s,r}^\pi}\right)+\oldwidecheck{\phi}_{u,t}U^\pi_{s,t}.
\end{multline*}

For bounding the last term of \eqref{eq:proof_decomp}, \eqref{eq:prop_lip_t_s_r} yields
\begin{align*}
\normlip{\phi_{u,t}\circ\phi_{t,s}\circ\phi_{s,r}^\pi
-\phi_{u,s}\circ\phi_{s,r}^\pi}\leq C\varpi(\omega_{s,u})\normlip{\phi_{s,r}^\pi}.
\end{align*}

Assuming that $r$, $s$, $t$ and $u$ belong to $\pi$ and combining
these inequalities and the fact the $\phi$ is stable (see Definition~\ref{def:stable_almost_flow}), 
\begin{equation*}
    U^\pi_{s,u}\leq U^\pi_{t,u}+(1+\delta_T)U^\pi_{s,t}+L^\pi (\phi^{\circledast}(C)(2+\delta_T)+C)\varpi(\omega_{s,u})
\end{equation*}
where 
\begin{equation*}
    L^\pi\eqdef \sup_{\substack{(s,t)\in\rTT^2\\s,t\in\pi}} \normlip{\phi_{t,s}^\pi}. 
\end{equation*}

For two successive points $s$ and $t\geq s$ of $\pi$
(See Definition~\ref{def:successivepoints}), $\phi^\pi_{t,r}=\phi_{t,s}\circ\phi^\pi_{s,r}$
so that $U^\pi_{s,t}=0$.

We assume that $T$ is small enough so that $\varkappa(1+\delta_T^2+\delta_T)<1$.
From Lemma~\ref{lem:davie}, 
\begin{equation}
    \label{eq:14}
    U^\pi_{s,t}\leq L^\pi\alpha_T\varpi(\omega_{s,t}) 
    \text{ with }\alpha_T\eqdef (2+\delta_T)\frac{(\phi^{\circledast}(C)(2+\delta_T)+C)}{1-\varkappa(1+\delta_T^2+\delta_T)}.
\end{equation}
In particular, for $r=s$, $U_{r,t}=\normlip{\phi^\pi_{t,r}-\phi_{t,r}}$.

Let us bound $L^\pi$. For this, with \eqref{eq:14} and (\ref{eq:h2})
with $\gamma=1$,
\begin{align*}
L^\pi &\leq \sup_{(s,t)\in\rTT^2}\normlip{\phi_{t,s}^\pi-\phi_{t,s}}
    +\max_{(s,t)\in\rTT^2}\normlip{\phi_{t,s}}\\
&\leq L^\pi\alpha_T\varpi(\omega_{0,T})+1+\delta_T.
\end{align*}
For $T$ small enough so that $\alpha_T\varpi(\omega_{0,T})\leq 1/2$, 
$L^\pi$ is uniformly bounded. Injecting this control of $L^\pi$ in \eqref{eq:14}, 
\begin{equation}
    \label{eq:13}
    \normlip{\phi^\pi_{t,r}-\phi_{t,r}}
    \leq K\varpi(\omega_{r,t}), \ \forall (r,t)\in\rTT^2,\ r,t\in\pi,
\end{equation}
for some constant $K$ that does not depend on the partition $\pi$.

It remains to establish \eqref{eq:13} for any pair of time $(r,t)\in\rTT^2$.

For this, let $t_\pi$ (resp. $r_\pi$) be the greatest (resp. smallest) point of $\pi$ 
below (resp. above) $t$ (resp. $r$). Then with the definition of
$\phi^\pi$ (Notation~\ref{not:iteratedproducts}),
\begin{equation*}
    \phi^\pi_{t,r}-\phi_{t,r}=\phi_{t,t_\pi}\circ \phi^\pi_{t_\pi,r}-\phi_{t,t_\pi}\circ\phi_{t_\pi,r}
    +\phi_{t,t_\pi,r}.
\end{equation*}
With \eqref{eq:prop_lip_t_s_r} and \eqref{eq:h2} of
Definition~\ref{def:almost_flow} with $\gamma=1$,
\begin{equation}
    \label{eq:16}
    \normlip{\phi^\pi_{t,r}-\phi_{t,r}}
    \leq (1+\delta_T)\normlip{\phi^\pi_{t_\pi,r}-\phi_{t_\pi,r}}+C\varpi(\omega_{r,t}).
\end{equation}
Similarly, 
\begin{equation*}
\phi^\pi_{t_\pi,r}-\phi_{t_\pi,r}=\phi^\pi_{t_\pi,r_\pi}\circ \phi^\pi_{r_\pi,r}-\phi_{t_\pi,r_\pi}\circ\phi_{r_\pi,r}
    +\phi_{t_\pi,r_\pi,r}.
\end{equation*}
Using (\ref{eq:sewing_sup}) and \eqref{eq:13},
\begin{multline}
    \label{eq:17}
    \normlip{\phi^\pi_{t_\pi,r}-\phi_{t_\pi,r}}
    \leq \normlip{\phi^\pi_{t_\pi,r_\pi}-\phi_{t_\pi,r_\pi}}\cdot\normlip{\phi_{r_\pi,r}}
    +C\varpi(\omega_{r,t})\\
    \leq (1+\delta_T)K\varpi(\omega_{r,t})+C\varpi(\omega_{r,t}).
\end{multline}
Inequality~\eqref{eq:prop_UL_4_points}, which is \eqref{eq:13} applied for any $(r,t)\in\rTT^2$,
stems from \eqref{eq:13}, \eqref{eq:16} and~\eqref{eq:17}.
\end{proof}

\begin{corollary}
\label{prop:stable_almost_flow_unique_flow}
If $\phi\in\cS\cA_{\delta_T,\varpi}(\uV)$ is a stable almost flow 
then there exists a unique Lipschitz flow $\psi$ in the galaxy containing $\phi$. Moreover, there is a constant $L\geq 0$ such that for all $(s,t)\in\uV$,
\begin{equation}
\label{eq:psi_phi_lip}
\normlip{\psi_{t,s}-\phi_{t,s}}\leq L\varpi(\omega_{s,t}).
\end{equation}
\end{corollary}

\begin{proof}
According to Theorem~\ref{prop:stable_almost_flow_UL}, $\phi$
satisfies the UL condition of Proposition~\ref{prop:ul_condition}.
Hence, it converges in the sup-norm to a Lipschitz flow $\psi\sim\phi$.
According to Proposition~\ref{prop:unique_flow},
$\psi$ is the only flow in the galaxy of $\phi$.

Passing to the limit in \eqref{eq:prop_UL_4_points} leads to \eqref{eq:psi_phi_lip}.
\end{proof}
 
\section{Perturbations}
\label{sec:perturbation_lip}
In \cite{brault1}, we have introduced the notion of perturbation of
almost flow. This notion still gives an almost flow. 

\subsection{Additive perturbations}

We recall that $\eta$, $\delta_T$ and $\gamma$ are defined in Section~\ref{sec:controls}.
\begin{definition}[Perturbation]
    \label{def:perturb}
    A \emph{perturbation} is an element $\epsilon\in\cF(\uV)$ such that 
     for any $(s,t)\in\rTT^2$ and $a,b\in\uV$,
    \begin{gather}
	\label{eq:epsilon:1}
	\epsilon_{t,t}=0,\\
	\label{eq:epsilon:2}
	\normsup{\epsilon_{t,s}}\leq C\varpi(\omega_{s,t}),\\
	\label{eq:epsilon:3}
	\abs{\epsilon_{t,s}(b)-\epsilon_{t,s}(a)}\leq \delta_T\abs{b-a}+ \eta(\omega_{s,t})\abs{b-a}^\gamma,
    \end{gather}
    where $\eta$ is defined by \eqref{eq:h_eta} and $C\geq 0$ is a constant.  
\end{definition}

\begin{proposition}[{\cite[Proposition~1]{brault1}}]
    \label{prop:pert:1}
    If $\phi\in\cF(\uV)$ is an almost flow and $\epsilon\in\cF(\uV)$
    is a perturbation, then $\psi\eqdef\phi+\epsilon$ is an almost flow
    in the same galaxy as $\phi$.
\end{proposition}

We introduce now the notion of Lipschitz perturbation, which 
is a perturbation on which a control stronger than \eqref{eq:epsilon:3}
holds.

\begin{definition}[Lipschitz perturbation]
\label{def:lip_perturbation}
    A \emph{Lipschitz perturbation} is a perturbation $\epsilon\in\cF^+(\uV)$  
with satisfies for a constant $C\geq 0$
\begin{equation}
\label{eq:epsi_lip}
\normlip{\epsilon_{t,s}}\leq C \varpi(\omega_{s,t}),\ \forall (s,t)\in\rTT^2.
\end{equation}
\end{definition}

Stable almost flows remain stable almost flows under Lipschitz perturbations.

\begin{proposition}[Stability of stable almost flow under Lipschitz perturbation]
    \label{prop:pert:2}
If $\phi\in\cS\cA_{\delta_T,\varpi}$ is a stable almost flow (see Definition~\ref{def:stable_almost_flow}) and $\epsilon$ is a Lipschitz perturbation, then $\psi\eqdef \phi+\epsilon$ is also a stable almost flow.
\end{proposition}

\begin{proof}
It is proved in Proposition \ref{prop:pert:1} that $\phi+\epsilon$ is an almost flow. Here we show that $\phi+\epsilon$ is a stable almost flow.

First, for any $a,b,c,d\in V$,
\begin{equation*}
|\epsilon_{t,s}(a)-\epsilon_{t,s}(b)-\epsilon_{t,s}(c)+\epsilon_{t,s}(d)|\leq 2C \varpi(\omega_{s,t})|a-c|\vee|b-d|,
\end{equation*}
so that $\epsilon_{t,s}$ satisfies a $\varpi$-compatible $4$-points
control (see Definition~\ref{def:varpi-compatible}) with $\widehat{\epsilon}_{t,s}\eqdef 2C \varpi(\omega_{s,t})$ and $\oldwidecheck{\epsilon}_{t,s}\eqdef 0$. Thus, $\phi+\epsilon$ satisfies a $\varpi$-compatible
$4$-points control with $\widehat{\phi+\epsilon}=\widehat{\phi}+2C\varpi(\omega_{s,t})$ and $\oldwidecheck{\phi_{t,s}+\epsilon_{t,s}}=\oldwidecheck{\phi_{t,s}}\leq 1+\delta_T$.

It remains to show that for any $(r,s,t)\in \rTT^3$,
$\normlip{\psi_{t,s,r}}\leq C\varpi(\omega_{r,t}),$
with $\psi_{t,s,r}\eqdef \psi_{t,s}\circ\psi_{s,r}-\psi_{t,r}$.
For any $a\in V$, we write
\begin{align*}
\psi_{t,s,r}(a)&=\underbrace{\left [\phi_{t,s}\circ (\phi_{s,r}+\epsilon_{s,r})(a)-\phi_{t,s}\circ\phi_{s,r}(a)\right]}_{\run_{r,s,t}(a)}
+\underbrace{\left[\epsilon_{t,s}\circ(\phi_{s,r}+\epsilon_{s,r})(a)-\epsilon_{t,s}\circ\epsilon_{s,r}(a)\right]}_{\rdeux_{r,s,t}(a)}\\
&+\underbrace{\phi_{t,s,r}(a)+\epsilon_{t,s,r}(a)}_{\rtrois_{r,s,t}(a)}.
\end{align*}
On the one hand, using the $\varpi$-compatible $4$-points control of
$\phi_{t,s}$, \eqref{eq:h2:gamma1},  (\ref{eq:epsilon:2}) and (\ref{eq:epsi_lip}) we write,
\begin{align*}
\normlip{\run_{r,s,t}}&\leq \widehat{\phi_{t,s}}(\normsup{\epsilon_{s,r}})(\normlip{\phi_{s,r}}+\normlip{\epsilon_{s,r}})+\oldwidecheck{\phi_{t,s}}\normlip{\epsilon_{s,r}}\\
&\leq \widehat{\phi_{t,s}}(C\varpi(\omega_{r,t}))(1+\delta_T+C\varpi(\omega_{0,T}))+(1+\delta_T)C\varpi(\omega_{r,t})\\
&\leq (C'(1+\delta_T+C\varpi(\omega_{0,T}))+1+\delta_T)\varpi(\omega_{r,t}),
\end{align*}
where $C'$ is a constant.

On the other hand, with \eqref{eq:h2}, \eqref{eq:epsi_lip}
\begin{multline*}
\normlip{\rdeux_{r,s,t}}\leq \normlip{\epsilon_{t,s}}(\normlip{\phi_{s,r}}+\normlip{\epsilon_{s,r}})+\normlip{\epsilon_{t,s}}\normlip{\epsilon_{s,r}}\\
\leq C\varpi(\omega_{r,t})(1+\delta_T+C\varpi(\omega_{0,T}))+C^2\varpi(\omega_{0,T})\varpi(\omega_{r,t})
\leq K_T \varpi(\omega_{r,t}), 
\end{multline*}
where $K_T\rightarrow 1$ when $T\rightarrow 0$.

Finally, with \eqref{eq:prop_lip_t_s_r}  and \eqref{eq:epsi_lip},
\begin{multline*}
\normlip{\rtrois_{r,s,t}}\leq \normlip{\phi_{t,s,r}}+\normlip{\epsilon_{t,r}}+\normlip{\epsilon_{t,s}}\normlip{\epsilon_{s,r}}
\leq (2C+C^2\varpi(\omega_{0,T}))\varpi(\omega_{r,t}).
\end{multline*}
This concludes the proof.
\end{proof}

Now, we prove another perturbation formula which is useful in Subsection~\ref{subsec:friz_victoir}.

We recall the $\delta_T, \eta$ and $\gamma$ are defined in Section~\ref{sec:controls}.

\begin{proposition}
\label{prop:perturbation_flow}
Let $\psi$ be a flow which may be decomposed as 
\begin{equation*}
    \psi_{t,s}(a)=\phi_{t,s}(a)+\epsilon_{s,t}(a),\ a\in\uV,\ (s,t)\in\TT^2
\end{equation*}
with for any $(s,t)\in\TT^2_+$ and $a,b\in\uV$, 
\begin{gather}
    \label{eq:p:1}
    \phi_{t,t}=\id,\ \epsilon_{t,t}=0,\\
    \label{eq:p:2}
    \abs{\phi_{t,s}(a)-\phi_{t,s}(b)}\leq (1+\delta_T)\abs{a-b}
    +\eta(\omega_{s,t})\abs{a-b}^\gamma,\ (a,b)\in\uV,\\
    \label{eq:p:3}
    \normsup{\phi_{t,s}-\id}\leq \delta_T,\\
    \label{eq:p:4}
    \normsup{\epsilon_{s,t}}\leq M\varpi(\omega_{s,t}).
\end{gather}
Then $\phi$ is an almost flow in the same galaxy as $\psi$. Besides, 
for any partition~$\pi$ of~$\TT$, 
\begin{equation}
    \label{eq:p:55}
    \normsup{\phi^\pi_{t,s}-\phi_{t,s}}\leq L\varpi(\omega_{s,t})
\end{equation}
where $L\leq 2\left[(3+\delta_T)M+\delta_T M^\gamma\right]/(1-(1+\delta_T)\kappa-\delta_T)$.
\end{proposition}

\begin{proof} 
    To show that $\phi$ is an almost flow, 
    it is sufficient to consider \eqref{eq:p:1}-\eqref{eq:p:4} 
    as well as controlling $\phi_{t,s,r}$. For $(r,s,t)\in\rTT^3$, 
    \begin{equation*}
	\psi_{t,s}\circ\psi_{s,r}(a)
	=
	\overbrace{\phi_{t,s}(\phi_{s,r}(a)+\epsilon_{s,r}(a))-\phi_{t,s}(\phi_{s,r}(a))}^{\run_{r,s,t}}
	+\phi_{t,s}(\phi_{s,r}(a))
	+\epsilon_{t,s}(\psi_{s,r}(a)).
    \end{equation*}
    Since $\psi$ is a flow, $\psi_{r,s,t}=0$ and then 
    \begin{equation*}
	\phi_{t,s,r}(a)=\run_{r,s,t}+\epsilon_{t,s}(\psi_{s,r}(a))-\epsilon_{t,r}(a).
    \end{equation*}
    With \eqref{eq:p:2},
    \begin{equation*}
	\abs{\run_{r,s,t}}\leq (1+\delta_T)M\varpi(\omega_{r,s})
	+\eta(\omega_{s,t})M^\gamma\varpi(\omega_{r,s})^\gamma.
    \end{equation*}
    With (\ref{eq:h_eta}), 
    \begin{equation*}
	\abs{\run_{r,s,t}}\leq  A \varpi(\omega_{r,t})
	\text{ with }A\eqdef (1+\delta_T)M+\delta_T M^\gamma.
    \end{equation*}
    It follows that $\normsup{\phi_{t,s,r}}\leq (2M+A)\varpi(\omega_{r,t})$.
    This proves that $\phi$ is an almost flow following Definition~\ref{sec:almost_flow}. The control \eqref{eq:p:55} follows from Theorem~\ref{thm:1}.
\end{proof}

\begin{corollary} 
    \label{cor:stability}
    Assume that $\uV$ is a finite-dimensional Banach space.
    Let $\Set{\psi^m}_{m\in\NN}$ be a family of flows 
    with decomposition $\psi^m=\phi^m+\epsilon^m$ where $(\phi^m,\epsilon^m)$
    satisfy \eqref{eq:p:1}-\eqref{eq:p:4} uniformly in $m$. Assume
    moreover that 
    \begin{equation}
	\label{eq:p:5}
	\normsup{\phi^m_{s,t}-\id}\leq \delta_{t-s},\ \forall (s,t)\in\rTT^2.
    \end{equation}
    Then any possible limit $\phi$ of $\phi^m_{t,s}$ (at least one exists) 
    satisfies \eqref{eq:p:1}-\eqref{eq:p:4}
    as well as~\eqref{eq:p:5} with the same constants.
\end{corollary}

\begin{proof} 
    With \eqref{eq:p:5}, Lemma~1 in \cite{brault1} can be applied
    uniformly. As \eqref{eq:p:2} is also uniform in $m$, this proves
    that for any $R>0$, $\Set{\phi^m(a)_{s,t}}_{(s,t)\in\rTT^2,\ a\in \overline{B(0,R)}}$ is equi-continuous
    where $\overline{B(0,R)}$ is the closed ball of center $0$ and some radius $R>0$.
    The Ascoli-Arzelà shows that at least one limit of $\phi^m$ exists. Clearly, 
    this limit satisfies the same properties as $\phi^m$.
\end{proof}


\subsection{Spatial perturbations}

We consider now another kind of perturbation.

Another kind of perturbation consists in, given an almost flow $\phi$, 

\begin{proposition}
    Let $\phi$ be an almost flow and $\epsilon$ be a perturbation satisfying Definition~\ref{def:perturb}.
    With $\psi_{t,s}:=\phi_{t,s}\circ(\id+\epsilon_{t,s})$ for any $(s,t)\in\rTT^2$, 
    $\psi$ is an almost flow which lives in the same galaxy as $\phi$.
\end{proposition}

\begin{proof}
    Clearly, $\psi_{t,t}=0$ since $\epsilon_{t,t}=0$. Thus, $\psi$ satisfies~\eqref{eq:h0}.
    With \eqref{eq:epsilon:2}, 
    \begin{equation*}
	\normsup{\psi_{t,s}-\id}
	\leq \normsup{\phi_{t,s}\circ(\id+\epsilon_{t,s})-\epsilon_{t,s}}
	+\normsup{\epsilon_{t,s}}
	\leq \delta_T+C\varpi(\omega_{0,T}).
    \end{equation*}
    Hence, $\psi$ satisfies \eqref{eq:h1}. 

    For any $a\in\uV$ and $(s,t)\in\rTT^2$, with \eqref{eq:h2},
    \eqref{eq:epsilon:2} and \eqref{eq:h_eta}, 
    \begin{equation}
	\label{eq:18}
	\abs{\psi_{t,s}(a)-\phi_{t,s}(a)}
	\leq 
	\abs{\epsilon_{t,s}(a)}+\eta(\omega_{t,s})\abs{\epsilon_{t,s}(a)}^\gamma
	\leq 
	(C+\delta_T)\varpi(\omega_{t,s}).
    \end{equation}

    Using \eqref{eq:h1} and \eqref{eq:epsilon:3}, for any $a,b\in\uV$,
    \begin{multline*}
	\abs{\psi_{t,s}(b)-\psi_{t,s}(a)}
	\leq
	\abs{\psi_{t,s}(b)-\phi_{t,s}(b)}
	+\abs{\phi_{t,s}(b)-\phi_{t,s}(a)}
	+\abs{\phi_{t,s}(a)-\psi_{t,s}(a)}
	\\
	\leq \abs{\epsilon_{t,s}(b)-\epsilon_{t,s}(a)}+(1+\delta_T)\abs{b-a}
	+\eta(\omega_{t,s})\abs{b-a}^\gamma.
    \end{multline*}
    With \eqref{eq:epsilon:3}, \eqref{eq:h2} is satisfied.

    Finally, using 
    \eqref{eq:h1}, \eqref{eq:epsilon:2} and \eqref{eq:h_eta}, 
    for any $a\in\uV$,
    \begin{multline}
	\label{eq:19}
	\abs{\psi_{t,s}\circ \psi_{s,r}(a)-\psi_{t,r}(a)}
	\\
	\leq 
	\abs{\psi_{t,s}\circ \psi_{s,r}(a)-\psi_{t,s}\circ\phi_{s,r}(a)}
	+\abs{\psi_{t,s}\circ \phi_{s,r}(a)-\phi_{t,s}\circ\phi_{s,r}(a)}
	+\abs{\phi_{t,s,r}(a)}
	+\abs{\phi_{t,r}(a)-\psi_{t,r}(a)}.
    \end{multline}
    With \eqref{eq:h2}, \eqref{eq:18} and \eqref{eq:h_eta},
    it is easily obtained that $\psi$ satisfies \eqref{eq:h3}.

    With \eqref{eq:18}, $\psi$ and $\phi$ belong to the same galaxy.
\end{proof}
\begin{proposition}[Stability of stable almost flow under Lipschitz spatial perturbation]
    \label{prop:spatial:pert:2}
    If $\phi\in\cS\cA_{\delta_T,\varpi}$ is a stable almost flow (see Definition~\ref{def:stable_almost_flow}) and $\epsilon$ is a Lipschitz perturbation, then $\psi\eqdef \phi\circ(\id\circ\epsilon)$ is also a stable almost flow.
\end{proposition}

\begin{proof}
    Since $\phi_{t,s}$ is Lipschitz with $\normlip{\phi_{t,s}}\leq 1+\delta_T$,  
    \eqref{eq:epsilon:2} implies that 
    \begin{equation}
	\label{eq:18:stable}
	\normsup{\psi_{t,s}-\phi_{t,s}}
	\leq (1+\delta_T)C\varpi(\omega_{s,t})
	\text{ for any }(s,t)\in\rTT^2.
    \end{equation}
    
For any $(s,t)\in\rTT^2$ and $a,b\in\uV$, 
\begin{multline*}
    \abs{\psi_{t,s}(b)-\phi_{t,s}(b)-\psi_{t,s}(a)+\phi_{t,s}(a)}
    \\
    \leq 
    \widehat{\phi_{t,s}}(\abs{\epsilon_{t,s}(a)}\vee\abs{\epsilon_{t,s}(b)})
    \cdot \abs{b+\epsilon_{t,s}(b)-a-\epsilon_{t,s}(a)}\vee\abs{b-a}
    +\phi_{t,s}^\circledast\abs{b+\epsilon_{t,s}(b)-a-\epsilon_{t,s}(a)}.
\end{multline*}
Since $\phi$ is a stable almost flow, with \eqref{eq:epsilon:2}
and \eqref{eq:epsi_lip}, 
\begin{multline}
    \label{eq:19bis}
    \abs{\psi_{t,s}(b)-\phi_{t,s}(b)-\psi_{t,s}(a)+\phi_{t,s}(a)}
    \\
    \leq 
    \phi_{t,s}^{\circledast}(C)\varpi(\omega_{s,t})
    (1+\normlip{\epsilon_{t,s}})\abs{a-b}
    +\phi_{t,s}^\dagger C\normlip{\epsilon_{t,s}}\abs{a-b}
    \leq 
    K\abs{a-b}
\end{multline}
for a constant $K$ that depends on $\delta_T$ and $C$ (here, 
both a control on~$\normsup{\epsilon_{t,s}}$ and~$\normlip{\epsilon_{t,s}}$).

Using a computation similar to \eqref{eq:19}, we then easily 
obtain from \eqref{eq:18:stable} and  \eqref{eq:19bis}
that~$\psi$ satisfies a $\varpi$-compatible $4$-points control
and satisfies \eqref{eq:prop_lip_t_s_r}.
\end{proof}


\section{Inversion of the flow}

In this section, we prove that our definition of stable almost
(Definition~\ref{def:stable_almost_flow}) flow is stable with respect to inversion.

\begin{proposition}
\label{prop:inversion_flow}
Let $\phi\in\cS\cT_{\delta_T,\varpi}$ be a stable almost flow
and $\psi$ the unique flow in the same galaxy as $\phi$ (Corollary~\ref{prop:stable_almost_flow_unique_flow}).
We assume that $\chi\eqdef \phi-\id$ satisfies a $4$-point control such that
\begin{align*}
\forall (s,t)\in\rTT^2,~\oldwidecheck{\chi_{t,s}}=\oldwidecheck{\phi_{t,s}}-1,
~\widehat{\chi_{t,s}}=\widehat{\phi_{t,s}},~\mathrm{and}~\normlip{\chi_{t,s}}\leq \delta_T.
\end{align*}
Then, for $T$ such that $\delta_T<1$, $\phi$ is invertible and $(\zeta_{s,t})_{(s,t)\in\rTT^2}\eqdef(\phi^{-1}_{t,s})_{(s,t)\rTT^2}$ is a  stable reverse almost flow which galaxy contains a unique flow which equal to $\psi^{-1}$.
\end{proposition}

\begin{proof}
According to item~\ref{lem:4pc:1:iv} of Lemma~\ref{lem:4pc:1} and because $\normlip{\chi_{t,s}}\leq \delta_T$ we know that for $T>0$ such that $\delta_T<1$, $\phi_{t,s}$ is invertible and that $\phi_{t,s}^{-1}$ satisfies a $4$-points control with
$\oldwidecheck{(\phi_{t,s}^{-1})}=1/(1-\oldwidecheck{\chi_{t,s}})$ and 
$\widehat{\phi_{t,s}^{-1}}(x)=\widehat{\chi_{t,s}}(\normlip{\phi_{t,s}}x)\normlip{\phi_{t,s}}$ for any $x\in\RR_+$. It follows that
$\oldwidecheck{(\phi_{t,s}^{-1})}\leq 1+\delta'_T$, with $\delta'_T\eqdef \delta_T/(1-\delta_T)$ and that $\phi_{t,s}$ satisfies a $\varpi$-compatible $4$-points control.

Moreover, $\normlip{\phi_{t,s}^{-1}}\leq 1/(1-\normlip{\chi_{t,s}})$ and
we assume $\normlip{\chi_{t,s}}\leq \delta_T$. It follows that
$\normlip{\phi_{t,s}^{-1}}\leq 1+\delta_T'$ which proves that \eqref{eq:h2} holds for $\phi^{-1}$. In substituting $a$ by $\phi^{-1}_{t,s}(a)$ in \eqref{eq:h1} we show that \eqref{eq:h1} holds for $\phi^{-1}$.

To prove that $(\zeta_{s,t})_{(s,t)\in\rTT^2}\eqdef (\phi^{-1}_{t,s})_{(s,t)\in\rTT^2}$ is a reverse stable almost flow, it remains to show that the conditions \eqref{eq:h3} and \eqref{eq:prop_lip_t_s_r} hold for any $(r,s,t)\in\rTT^3$.
Firstly, we compute with \eqref{eq:h3}, since $\phi_{t,s}\circ\phi_{s,r}$ is one-to-one,  
\begin{multline*}
    \normsup{\phi_{s,r}^{-1}\circ\phi_{t,s}^{-1}\circ\phi_{t,s}\circ\phi_{s,r}-\phi^{-1}_{t,r}\circ\phi_{t,s}
	\circ\phi_{s,r}}=\normsup{\phi_{t,r}^{-1}\circ\phi_{t,r}-\phi_{t,r}^{-1}\circ\phi_{t,s}
    \circ\phi_{s,r}}\\
    \leq (1+\delta_T')\normsup{\phi_{t,r}-\phi_{t,s}\circ\phi_{s,r}}
\leq 
M\varpi(\omega_{r,t}),
\end{multline*}
which yields with to $\normsup{\zeta_{r,s}\circ\zeta_{s,t}-\zeta_{r,t}}\leq M\varpi(\omega_{r,t})$.

Secondly, for any $a,b\in\uV$ and $(r,s,t)\in\rTT^3$, we set $a'\eqdef \phi_{s,r}^{-1}\circ\phi_{t,s}^{-1}(a)$, $b'\eqdef \phi_{s,r}^{-1}\circ\phi_{t,s}^{-1}(b)$, and
\begin{align*}
\Phi_{r,s,t}&\eqdef (\phi_{s,r}^{-1}\circ\phi_{t,s}^{-1}-\phi_{t,r}^{-1})\circ\phi_{t,s}\circ\phi_{s,r}(b')-(\phi_{s,r}^{-1}\circ\phi_{t,s}^{-1}-\phi_{t,r}^{-1})\circ\phi_{t,s}\circ\phi_{s,r}(a')\\
&=\phi_{t,r}^{-1}\circ\phi_{t,r}(b')-\phi_{t,r}^{-1}\circ\phi_{t,s}\circ\phi_{s,r}(b')-\phi_{t,r}^{-1}\circ\phi_{t,r}(a')+\phi_{t,r}^{-1}\circ\phi_{t,s}\circ\phi_{s,r}(a').
\end{align*}
We know that $\phi^{-1}_{t,r}$ satisfies a $\varpi$-compatible
$4$-points control and we use \eqref{eq:prop_lip_t_s_r}, 
\begin{align*}
\abs{\Phi_{r,s,t}}&\leq \widehat{\phi^{-1}_{t,r}}(\normsup{\phi_{t,s,r}})
\left[\normlip{\phi_{t,r}}\vee\normlip{\phi_{t,s}\circ\phi_{s,r}}\right]|b'-a'|
+\oldwidecheck{(\phi_{t,r}^{-1})}\normlip{\phi_{t,s,r}}|b'-a'|\\
&\leq \phi^{-1, \circledast}(M)\varpi(\omega_{r,t})(1+\delta_T)^2\abs{b'-a'}+(1+\delta_T)C\varpi(\omega_{r,t})\abs{b'-a'}.
\end{align*}
Then substituting $a'$ and $b'$ by  $\phi_{s,r}^{-1}\circ\phi_{t,s}^{-1}(a)$
and $\phi_{s,r}^{-1}\circ\phi_{t,s}^{-1}(a)$, 
\begin{multline*}
\normlip{\phi_{s,r}^{-1}\circ\phi_{t,s}^{-1}-\phi_{t,r}^{-1}} 
\leq [\phi^{-1, \circledast}(M)\varpi(\omega_{r,t})(1+\delta_T)^2+(1+\delta_T)C\varpi(\omega_{r,t})]\normlip{\phi_{s,r}^{-1}\circ\phi_{t,s}^{-1}}\\
\leq
\left[\phi^{-1, \circledast}(M))(1+\delta_T)^2+(1+\delta_T)C\right](1+\delta'_T)^2\varpi(\omega_{r,t}).
\end{multline*}
Hence $\zeta$ is a stable reverse almost flow. According to Corollary~\ref{prop:stable_almost_flow_unique_flow}, $\zeta^\pi$ converges to a unique Lipschitz flow $\zeta^{\infty}$ in $\cF(\uV)$. But, $\zeta_{s,t}^\pi=(\phi^\pi_{t,s})^{-1}$, which yields to 
$\zeta_{s,t}^\pi\circ\phi^\pi_{t,s}=\id$ and passing to limit
$\zeta_{s,t}^\infty\circ\psi_{t,s}=\id$. This concludes the proof.
\end{proof}



\section{Generalized solution to rough differential equations}
\label{sec:generalized_solution_rde}
Almost flows approximate of flows, similarly to numerical algorithms.
In classical analysis, flows are strongly related to solutions of ordinary differential
equations (ODE). Rough differential equations (RDE) were solved first using 
fixed point theorems  on paths \cite{lyons98a}. The technical difficulty with this 
approach is that the solution itself should be a rough path. 

Later, A.M. Davie introduced in \cite{davie05a} another notion of solution of RDE
which no longer involves solutions as rough paths, but only as paths. 
We abstract here this approach in order to relate almost flows and paths.

\begin{definition}[Generalized solution in the sense of Davie]
\label{def:solution_davie}
    Let $\phi$ be an almost flow. 
    Let $a\in\uV$ and $r\in\TT$.  A $\uV$-valued path $\Set{y_{r \curvearrowright t}}_{(r,t)\in\TT^2}$
    is said to be a \emph{solution in the sense of Davie} of $dy=\phi_{\vd t}(y)$ with the initial condition $a$ at time $r$ 
    if $y_{r\curvearrowright r}=a$ and there exists a constant $C$ such that 
    \begin{equation}
	\label{eq:davie:5}
	\abs{y_{r\curvearrowright t}-\phi_{t,s}(y_{r\curvearrowright s})}
	    \leq C\varpi(\omega_{s,t}),\ \forall r\leq s\leq t\leq T.
    \end{equation}
\end{definition}

\begin{definition}[Manifold of solutions]
    A family $\Set{y_{r\curvearrowright \cdot}(a)}_{r\in\TT,\ a\in\uV}$ of solutions
    satisfying \eqref{eq:davie:5} and $y_{r\curvearrowright r}(a)=a$ is called
    a \emph{manifold of solutions}. We write $\vd y=\phi_{\vd t}(y)$ to denote the whole
    family of solutions.
\end{definition}

\begin{definition}[Lipschitz manifold of solutions]
    If $a\mapsto y_{r\curvearrowright \cdot}(a)$ is uniformly 
    Lipschitz continuous from $(\uV,d)$ to $(\cC([r,T],\uV),\normsup{\cdot})$, 
    then we say that the manifold of solutions is \emph{Lipschitz}.
\end{definition}

\begin{remark} When $\phi_{t,s}=\id+\chi_{t,s}$, 
    then \eqref{eq:davie:5} may be written 
	$\abs{y_{s,t}-\chi_{t,s}(y_s)}\leq C\varpi(\omega_{s,t})$
    with $y_{s,t}\eqdef y_{r\curvearrowright t}-y_{r\curvearrowright s}$.
    for $(r,s,t)\in\TT^+_2$
    This is the form used by A.M. Davie in~\cite{davie05a}. 
\end{remark}

Flows and manifold of solutions are  closely related. 
Besides, a manifold of solutions is in relation with a whole galaxy.
The proof of the next lemma is immediate so  we skip it.

\begin{lemma}
    \label{lem:rde:1}
    A flow $\psi$ generates a manifold of solutions
    to  $\vd y=\psi_{\vd t}(y)$
    through $y_{r\curvearrowright t}(a)\eqdef \psi_{t,r}(a)$, 
    $(r,t)\in\TT^+_2$, $a\in\uV$.
    Besides, $y$ is also solution to   $\vd y=\phi_{\vd t}(y)$ 
    for any almost flow $\phi$ in the galaxy containing $\psi$.
\end{lemma}

\subsection{Existence of a flow from a family of solutions}

First,  we show how to construct a flow from a suitable
family of paths.

\begin{proposition}
    \label{prop:davie:1}
    Consider an almost flow $\phi$ and $T$ small enough.
    Assume that there exists a family $\Set{y_{0\curvearrowright t}(a)}_{t\in\TT,a\in\uV}$
    of $\uV$-valued paths, continuous in time and Lipschitz continuous in space such that  
    such that
    \begin{gather*}
	y_{0\curvearrowright 0}=\id,\ 
	\normsup{y_{0\curvearrowright t}-\phi_{t,s}(y_{0\curvearrowright s})}
	\leq C\varpi(\omega_{s,t}), \ \forall (s,t)\in\rTT^2,
	\\
	\sup_{t\in\TT} \{\normlip{y_{0\curvearrowright t}-\id}+
	\normsup{y_{0\curvearrowright t}-\id}\}\leq K_T
	\text{ where }K_T\xrightarrow[T\to 0]{} 0.
    \end{gather*}
    Then  $\Set{y_{0\curvearrowright t}(a)}_{t\in\TT,a\in\uV}$ 
    is a family of Lipschitz diffeomorphisms on $\uV$
    
    Moreover,
with $\psi_{t,s}(a)\eqdef y_{0\curvearrowright t}\circ y_{0\curvearrowright s}^{-1}(a)$
for $(s,t)\in\TT^2$ and $a\in\uV$, 
$\Set{\psi_{t,s}(a)}_{(s,t)\in\TT^2,a\in\uV}$ 
is also a family of Lipschitz diffeomorphisms and defines a flow in the same galaxy as $\phi$.
\end{proposition}
\begin{proof}
    The Lipschitz inverse mapping shows that $y_{0\curvearrowright t}$ is invertible with a
    Lipschitz continuous inverse $y^{-1}_{0\curvearrowright t}$ when $K_T<1$  (\cite{abraham} p.~124).
    
   Assuming that $T$ is small enough, 
 $\psi_{t,s}(a)\eqdef y_{0\curvearrowright t}\circ y_{0\curvearrowright s}^{-1}(a)$
    for any $(s,t)\in\TT^2$ and $a\in\uV$ defines an invertible flow.
   
    Besides, for any $(s,t)\in\TT^2_+$, $\psi_{s,t}$ is Lipschitz continuous 
    since both $y_{0\curvearrowright t}$ and $y_{0\curvearrowright s}^{-1}$ are
    Lipschitz continuous.

    It remains to prove that $\phi\sim\psi$. For $a\in\uV$, let us
    set $b\eqdef y_{0\curvearrowright s}^{-1}(a)$. Thus,
    \begin{multline*}
	\abs{\psi_{t,s}(a)-\phi_{t,s}(a)}
	=\abs{\psi_{t,s}(y_{0\curvearrowright s}(b))-\phi_{t,s}(y_{0\curvearrowright s}(b))}\\
	\leq
	\abs{\phi_{t,s}(y_{0\curvearrowright s}(b))-y_{0\curvearrowright t}(b)}
	\leq C\varpi(\omega_{s,t}). 
    \end{multline*}
    Thus, $\psi\sim\phi$.
\end{proof}

\subsection{Uniqueness and continuity of a solution in the sense of Davie}

A stable almost flow $\phi$ satisfies the condition UL (see Theorem~\ref{prop:stable_almost_flow_UL}), so 
that there exists a unique flow $\psi$ in the same galaxy as $\phi$. Furthermore, 
$\psi$ is Lipschitz. 

The flow $\psi$ generates a manifold of solutions. We show that
there exists only one such manifold with a Lipschitz continuity 
result. Note that in the following proposition, $\zeta$
is not assumed to be stable.

\begin{proposition}
    \label{prop:rde:1}
Let $\phi$ be a stable almost flow and $\zeta$ be an almost flow. 

Let $y$ and $z$ be two paths from $[0,T]$ to $\uV$ such that 
\begin{equation*}
    \abs{y_t-\phi_{t,s}(y_s)}\leq K\varpi(\omega_{s,t})
    \text{ and }
    \abs{z_t-\zeta_{t,s}(z_s)}\leq K\varpi(\omega_{s,t})
    ,\ \forall (s,t)\in\rTT^2.
\end{equation*}

Let us write $\alpha_{t,s}\eqdef \zeta_{t,s}-\phi_{t,s}$ 
and $\alpha_{t,s,r}\eqdef \zeta_{t,s,r}-\phi_{t,s,r}$. 
Let $\epsilon_1, \epsilon_2, \epsilon_3>0$ be such that 
for any $(r,s,t)\in\rTT^3$, 
\begin{equation*}
    \abs{\alpha_{t,s,r}(z_r)}\leq \epsilon_1\varpi(\omega_{r,t}),\ 
    \normlip{\alpha_{s,t}}\leq \epsilon_2
    \text{ and }
    \abs{\alpha_{s,t}(z_s)}\leq \epsilon_3.
\end{equation*}

Then there exists a time $T$ small enough and a constant $C$ that depends only on 
$\phi$, $K$, $T\mapsto \delta_T$, $\varkappa$ and $\sup_{(r,s,t)\in\TT^3}\normlip{\phi_{t,s,r}}/\varpi(\omega_{r,t})$ such that 
\begin{gather*}
    \abs{y_t-z_t}\leq C(\epsilon_1+\epsilon_2+\epsilon_3+\abs{y_0-z_0})\\
    \text{and }
    \abs{y_t-\phi_{t,s}(y_s)-z_t+\zeta_{t,s}(z_s)}
    \leq C(\epsilon_1+\epsilon_2+\epsilon_3+\abs{y_0-z_0})\varpi(\omega_{s,t})
\end{gather*}
for all $(s,t)\in\TT^2$.
\end{proposition}

Corollary~\ref{cor:manifold} below is immediate by applying $\phi=\zeta$ 
to Proposition~\ref{prop:rde:1}. Its proof relies of the continuous time version of Lemma~\ref{lem:davie}
which we now state.

\begin{lemma}[{Continuous time version of Lemma~\ref{lem:davie}}]
    \label{lem:davie:CT}
Let us consider a family $U\eqdef \Set{U_{s,t}}_{s,t\in,\TT^2}$ with values in $\RR_+$ satisfying
 for any $(r,s,t)\in\TT^3,$ 
\begin{align}
\label{eq:n1:CT}
U_{r,s}&\leq E\varpi(\omega_{r,s}),\\
\notag
U_{r,t}&\leq(1+\alpha_T)U_{r,s}+(1+\alpha_T)U_{s,t}+B\varpi(\omega_{r,t}), 
\end{align}
for some constants $E\geq 1$, $B\geq 0$ and $\alpha_T\geq 0$ that decreases to $0$ as $T\to 0$.
Then for any $T$ such that $\varkappa(1+\alpha_T)<1$, 
    \begin{gather}
	\notag
	U_{r,t}\leq A\varpi(\omega_{r,t}),\ \forall (r,t)\in\TT^2,\\
	\label{eq:davie:4:cst:CT}
	\text{ with }A\eqdef B\frac{(2+\alpha_T)}{1-(\varkappa(1+\alpha_T)^2+\delta_T)}.
    \end{gather}
    In particular, the choice of $A$ in \eqref{eq:davie:4:cst:CT}
    does not depend on the bound $E$ in \eqref{eq:n1:CT}.
\end{lemma}
\begin{proof}
    The proof is similar as the one of Lemma~7 in \cite{brault1}
    from Eq.~(31).
\end{proof}

\begin{corollary}
    \label{cor:manifold}
    If $\phi$ is a stable almost flow, there exists 
    one and only one manifold of solutions to  $\vd y=\phi_{\vd t}(y)$.
    Besides, this manifold of solutions is Lipschitz.
\end{corollary}

\begin{remark}
\label{rem:unique_flow}
As seen in Lemma~\ref{lem:rde:1}, the 
notion of manifold of solution is associated to a galaxy. 
Hence, a galaxy with a stable almost flow is associated
to a unique manifold of solutions (actually, 
we have not proved that if $\phi$ is a stable almost flow, 
then the associated flow is also stable). 
\end{remark}

\begin{proof}[Proof of Proposition~\ref{prop:rde:1}]
We define
\begin{equation*}
    V_{s,t}=\abs{z_{t}-\zeta_{t,s}(z_s)-y_t+\phi_{t,s}(y_s)},\ \forall(s,t)\in\TT^2.
\end{equation*}
Clearly, $V_{s,t}\leq 2K\varpi(\omega_{s,t})$. 

For any $(r,s,t)\in\TT^3$, 
\begin{multline*}
    z_t-\zeta_{t,r}(z_r)-y_t+\phi_{t,r}(y_r)\\
    = z_t-\zeta_{t,s}(z_s)-y_t+\phi_{t,s}(y_s)
    +
    \zeta_{t,s}(z_s)-\zeta_{t,s}(\zeta_{s,r}(z_r))\\
    -\phi_{t,s}(y_s)+\phi_{t,s}(\phi_{s,r}(y_r))
    +\zeta_{t,s,r}(z_r)-\phi_{t,s,r}(y_r)\\
    =
    z_t-\zeta_{t,s}(z_s)-y_t+\phi_{t,s}(y_s)
    +\alpha_{t,s}(z_s)-\alpha_{t,s}(\zeta_{s,r}(z_r))\\
    +\phi_{t,s}(z_s)-\phi_{t,s}(\zeta_{s,r}(z_r))
    -\phi_{t,s}(y_s)+\phi_{t,s}(\phi_{s,r}(y_r))\\
    +\alpha_{t,s,r}(z_r)+\phi_{t,s,r}(z_r)-\phi_{t,s,r}(y_r).
\end{multline*}
Set $L\eqdef\sup_{(r,s,t)\in\TT^3}\normlip{\phi_{t,s,r}}$.
With the $4$-points control
of $\phi_{s,t}$, 
\begin{multline*}
    V_{t,r}
    \leq V_{t,s}
    +\oldwidecheck{\phi}_{t,s}V_{r,t}
	+L\abs{z_r-y_r}\varpi(\omega_{r,t})
    \\
    +\widehat{\phi}_{t,s}(\abs{z_s-\zeta_{s,r}(z_r)}\vee\abs{y_s-\phi_{s,r}(y_r)})
	\cdot (\abs{z_s-y_s}\vee\abs{\zeta_{s,r}(z_r)-\phi_{s,r}(y_r)})\\
	+\abs{\alpha_{t,s,r}(z_r)}
	+\normlip{\alpha_{t,s}}\abs{z_s-\zeta_{s,r}(z_r)}.
\end{multline*}

Since $\oldwidecheck{\phi}_{t,s}\leq 1+\delta_T$ and $\normlip{\phi_{t,s}}\leq 1+\delta_T$, 
\begin{multline*}
    V_{t,r}
	\leq V_{t,s}+(1+\delta_T)V_{r,t}
	+L\normsup{z-y}\varpi(\omega_{r,t})\\
	+\widehat{\phi}_{t,s}(K\varpi(\omega_{s,r}))\big((1+\delta_T)\normsup{z-y}+\epsilon_3)
	+(\epsilon_1+\epsilon_2K)\varpi(\omega_{r,t}).
\end{multline*}
Since $\widehat{\phi}_{t,s}$ is $\varpi$-compatible (see Defintion~\ref{def:varpi-compatible}), 
$\widehat{\phi}_{t,s}(K\varpi(\omega_{r,s}))\leq \Phi(K)\varpi(\omega_{r,t})$ so that
\begin{gather}
    \label{eq:rde:12}
    V_{r,t}\leq V_{s,t}+ (1+\delta_T) V_{r,t}+B\varpi(\omega_{r,t})\\
    \label{eq:rde:11}
    \text{ with }B\eqdef (L+(1+\delta_T)\Phi(K))\normsup{y-z}+\epsilon_1+\epsilon_2 K+\epsilon_3 \Phi(K)).
\end{gather}
Owing to \eqref{eq:rde:12}-\eqref{eq:rde:11}, 
from  Lemma~\ref{lem:davie:CT}, 
for $T$ small enough (depending only on $\varkappa$ and $T\mapsto \delta_T$), 
for all $(r,t)\in\TT^2$, 
\begin{equation*}
    V_{r,t}\leq BC\varpi(\omega_{r,t}) 
\text{ with  } 
    C\eqdef \frac{2+\delta_T}{1-(\varkappa(1+\delta_T)^2+\delta_T)}.
\end{equation*}

For any $t\in[0,T]$,  
since $\normlip{\phi_{t,0}}\leq 1+\delta_T$, 
\begin{multline*}
    \abs{y_t-z_t}
    \leq \abs{y_t-z_t-\phi_{t,0}(y_0)+\zeta_{t,0}(z_0)}
    +\abs{\phi_{t,0}(z_0)-\zeta_{t,0}(z_0)}
    +\abs{\phi_{t,0}(z_0)-\phi_{t,0}(y_0)}
    \\
    \leq 
    V_{0,t}+\epsilon_3+(1+\delta_T)\abs{y_0-z_0}.
\end{multline*}
With the expression of $B$ in \eqref{eq:rde:11}, for any $t\in[0,T]$, 
we see that there exists constants $A$ and $A'$ that 
depend only on $\varkappa$, $L$, $\delta_T$, $K$ and $\Phi(K)$ such that 
\begin{equation*}
    \abs{y_t-z_t}\\
    \leq 
    A\normsup{y-z}\varpi(\omega_{0,T})
    +A'(\epsilon_1+\epsilon_2+\epsilon_3)\varpi(\omega_{0,T})
    +\epsilon_3
    +(1+\delta_T)\abs{y_0-z_0}.
\end{equation*}
Choosing $T$ small enough so that $A\varpi(\omega_{0,T})\leq 1/2$ implies that
\begin{equation*}
    \normsup{y-z}
    \leq 
    2A'(\epsilon_1+\epsilon_2+\epsilon_3)\varpi(\omega_{0,T})
    +2\epsilon_3
    +2(1+\delta_T)\abs{y_0-z_0}.
\end{equation*}
This concludes the proof.
\end{proof}

\section{Application to Rough differential equation}
\label{sec:applications}

In this section, we show how our framework allows us to link the different flow
based approaches. The key is to show that Friz-Victoir's and Bailleul's almost
flows are different perturbations of the Davie's almost flow. 
Friz-Victoir's approach is limited to geometric rough paths of regularity $2\leq p<3$
in finite dimensional spaces (see however \cite{grong} for a recent extension to infinite dimensional space). 
Bailleul's one work even in 
an infinite dimensional Banach space, while being restricted
\textit{a priori} to geometric rough paths. 
For rough paths of regularity $2\leq p<3$, Davie's approach can be used 
for geometric rough paths as well as non-geometric ones. To deal 
with non-geometric rough paths with Friz-Victoir's and Bailleul's approaches, 
one may use the results of \cite{lejay-victoir} for a regularity $p\in(2,3]$
and \cite{hairer-kelly} for $p\geq 3$.
As shown in \cite{bailleul18,lejay20}, using formal logarithhms
could be applied in broader structures than tensor algebras, such as algebras
of root trees. This also allows to consider non-geometric rough paths.

We did not recall notions of rough path theory. The reader can find
a clear introduction in \cite{lejay03,lyons07a,friz14a,friz}.
 We start by giving some notations.

  In  this section, the remainder  introduced in
  Section~\ref{sec:controls} is of the type
  \begin{align*}
    \varpi(\tau):=\tau^{(2+\gamma)/p}, ~\forall \tau>0,
  \end{align*}
  with  $\gamma\in (0,1]$ and a real number $p>0$
  satisfying $2+\gamma>p$.

\subsection{Rough path notations}
Before showing the link between the different based flow approaches,
we set notations of classical objects of  rough path theory.

Given another Banach space $(\uU,\abs{\cdot})$ and a real number $p\geq 1$, let us denote by 
$\cpC(\TT,\uU)$ the space of $1/p$-Hölder paths controlled by $\omega$, which we equip we the semi-norm
\begin{align*}
\normp{x}\eqdef \sup_{(s,t)\in\rTT^2, s\neq t}\frac{\abs{x_{s,t}}}{\omega_{s,t}^{1/p}},
\end{align*} 
this quantity being bounded by definition.

We define also $\cpV([s,t],\uU)$ the space of bounded $p$-variation paths
from $[s,t]$ to $\uU$ which we equip with the $p$-variation semi-norm on $[s,t]$ denoted by $\normpa{[s,t]}{x}$.

Moreover, if $x\in\cpC([0,T],\uU)$, then $x\in\cpV([s,t],\uU)$  and
\begin{align*}
\normpa{[s,t]}{x}\leq \normp{x}\omega_{s,t}^{1/p}.
\end{align*}

We denote by $\lfloor\cdot\rfloor$ the floor function.

For an integer $N\geq \lfloor p\rfloor$, let $\cT^{p,N}(\uU)$ be
the space of $1/p$-Hölder rough path controlled by $\omega$ of order $N$. If $\bx\in\cT^{p,N}(\uU)$ we denote by $\bx^{(k)}$ the component of $\bx$ in
$\uU^{\otimes k}$ with $0\leq k\leq N$ an integer and
$S_j(\bx)\eqdef \sum_{j=0}^k\bx^{(j)}$. Obviously, $\bx=S_N(\bx)$. We denote the homogeneous semi-norm
 \begin{align*}
 \normp{\bx}\eqdef \sup_{k\leq N}\sup_{[s,t]\in\TT^2, s\neq t}\frac{\abs{\bx^{(k)}_{s,t}}}{\omega_{s,t}^{k/p}},
 \end{align*}
which is finite by definition. Moreover we set $\cT^p(U)\eqdef \cT^{p,\lfloor p\rfloor}(U)$.

For $N\geq 0$, we denote $G^N(\uU)$ the free nilpotent group
(Chapter~$7$ in \cite{friz}).

Let $\cG^p(\uU)\eqdef \cpC([0,T],G^{\lfloor p \rfloor}(\uU))$ be the set of weak-geometric rough paths of finite $1/p$-Hölder rough path controlled by $\omega$ with values in $\uU$.

 When $\uU=\RR^\ell$ ($\ell\geq 0$ an integer). For any multi-indice  $I\eqdef (i_1,\dots,i_k)\in\{1,\dots,\ell\}^k$ we set $\abs{I}\eqdef k$ and $e_I\eqdef e_{i_1}\otimes\dots\otimes e_{i_k}$
 where $\Set{e_1,\dots,e_\ell}$ is the canonical basis of $\RR^\ell$. If $\bx\in\cT^{p}(\RR^\ell)$. If $\bx\in\cT^{p}(\RR^\ell)$, then
$\bx^I$ denote the coordinates of $\bx^{(k)}$ in the basis $(e_I)_{\abs{I}=k}$. It follows that
$S_k(\bx)=\sum_{\abs{I}\leq k}\bx^Ie_I$. If $x\in\cbV(\TT,\RR^l)$ then for any integer $N\geq 0$,
\begin{align*}
S_N(x)=\sum_{\abs{I}\leq N} x^Ie_I,
\end{align*}
where $x^I_{s,t}\eqdef \int_{s\leq t_k\leq \dots\leq t_1\leq t}\vd{x}_{t_k}^{i_k}\dots\vd x^{i_1}_{t_1}$.

\subsection{Davie's approach}
\label{sec:applications_davie}

Let us consider now a $p$-rough path $\bx\in\cT^p(\uU)$ with $2\leq p<3$ for a Banach space~$\uU$. 
A Rough Differential Equation (RDE) is a solution $y$ taking its values in another Banach space $\uV$ to 
\begin{equation}
    \label{eq:rde:1}
y_t=a+\int_s^t f(y_u)\vd \bx_u,~\forall (s,t)\in\rTT^2, 
\end{equation}
provided that $f:\uV\to L(\uU,\uV)$ is regular enough. 

Existence of solution to \eqref{eq:rde:1} was proved first by T.~Lyons
using a Picard fixed point theorem \cite{lyons98a}. In \cite{davie05a}, A.M. Davie
provided an alternative approach based on Euler-type schemes. Over the approach
of T. Lyons, it has the advantage that the solution is thought of as a path 
with values in $\uV$ and not in the tensor space $\uT_2(\uU\oplus\uV)$. 

For $f\in\cC_b^1(\uV,L(\uU,\uV))$, we define
\begin{equation}
    \label{eq:rde:2}
    \phi_{t,s}(a)=a+f(a)\bx^{(1)}_{s,t}+\vd f(a)\cdot f(a)\bx^{(2)}_{s,t},
\end{equation}
where $\vd f(a)$ is the differential of $f$ in $a$.
Definition~\ref{def:solution_davie} coincides with the one defined 
by A.~M.~Davie in \cite{davie05a} for the notion of solution to~\eqref{eq:rde:1}.

The following lemma is a generalization in an infinite dimensional setting of
Theorems~3.2 and~3.3 in \cite{davie05a} with a bounded function $f$.

\begin{lemma}
\label{lem:davie_almost_stable_flow}
Let $p\in\itvco{2,3}$ and $\gamma>p-2$. 
\begin{enumerate}[thm]
    \item\label{lem:davie_almost_stable_flow:i} If $f\in \cC^{1+\gamma}_b{(\uV,L(\uU,\uV))}$, then the family $\phi$ defined by~\eqref{eq:rde:2} is an almost flow.
    \item\label{lem:davie_almost_stable_flow:ii} If $f$ is of class $\cC^{2+\gamma}_b{(\uV,L(\uU,\uV))}$, then $\phi$ is a stable almost flow.
\end{enumerate}
\end{lemma}
\begin{proof}
According to Proposition~5 in \cite{brault1}, $\phi$ is an almost flow
as soon as $f\in \cC^{1+\gamma}_b{(\uV,L(\uU,\uV))}$ with $2+\gamma>p$.

We now assume that $f\in\cC^{2+\gamma}_b{(\uV,L(\uU,\uV))}$.  We show that $\phi$
verifies a $\varpi$-compatible $4$-points control (see Definition~\ref{def:varpi-compatible}).
For any $a,b,c,d\in\uV$, $(s,t)\in\TT^2$, we compute
\begin{multline}
\phi_{t,s}(a)-\phi_{t,s}(b)-\phi_{t,s}(c)+\phi_{t,s}(d)
=a-b-c+d+\underbrace{[f(a)-f(b)-f(c)+f(d)]\bx^{(1)}_{s,t}}_{\run_{s,t}'}\\
+\underbrace{[\vd f\cdot f(a)-\vd f\cdot f(b)-\vd f\cdot f(c)+\vd f\cdot f(d)]\bx^{(2)}_{s,t}}_{\rdeux'_{s,t}}.
\label{eq:phi_rough_4_points}
\end{multline}
Then, we apply Lemma~\ref{lem:4_points_f} to $f\in \cC^2_b{(\uV,L(\uU,\uV))}$ and to  $\vd f\cdot f\in \cC^{1+\gamma}_b{(\uV,L(\uU^{\otimes 2},\uV))}$ to obtain
\begin{multline}
\label{eq:run'}
|\run_{s,t}'|
\leq\normp{\bx^{(1)}}\omega_{s,t}^{1/p} 4\normsup{\vd^2f}(|a-c|\vee |b-d|)^2|a-b|
\\
+\normp{\bx^{(1)}}\omega_{0,T}^{1/p}\normsup{\vd f}|a-b-c+d|,
\end{multline}
and
\begin{multline}
\label{eq:rdeux'}
|\rdeux_{s,t}'|
\leq\normpp{\bx^{(2)}} \omega_{s,t}^{2/p}\Bigg[2\normf{\gamma}{\vd(\vd f\cdot f)}\left(|a-c|\vee |b-d|\right)^\gamma|a-b|\\
+\normsup{\vd(\vd f\cdot f)}|a-b-c+d|\Bigg].
\end{multline}
Combining \eqref{eq:phi_rough_4_points}, \eqref{eq:run'} and \eqref{eq:rdeux'} we set for all $y\in\RR_+$, $\widehat{\phi_{t,s}}(y)\eqdef c'_1\left[\omega_{s,t}^{1/p} y^2+\omega_{s,t}^{2/p}y^\gamma\right]$
where $c'_1\eqdef \normp{\bx^{(1)}}4\normsup{\vd^2f}+\normpp{\bx^{(2)}} 2\normf{\gamma}{\vd(\vd f\cdot f)}$ and
\begin{align*}
\oldwidecheck{\phi_{t,s}}\eqdef 1+\normp{\bx^{(1)}}\omega_{0,T}^{1/p}+\normsup{\vd(\vd f\cdot f)}\normpp{\bx^{(2)}} \omega_{0,T}^{2/p}.
\end{align*}
It follows that $\oldwidecheck{\phi_{t,s}}\leq 1+\delta_T$ and that $\phi_{t,s}$ is $\varpi$-compatible. Indeed, for $\alpha\in\RR_+$,
\begin{multline*}
\widehat{\phi_{t,s}}\left(\alpha\omega_{s,t}^{(2+\gamma)/p}\right)
    \leq c'_1(\alpha^2\vee \alpha^\gamma)\left(\omega_{s,t}^{(5+2\gamma)/p}
+\omega_{s,t}^{(2+2\gamma+\gamma^2)/p}\right)\\
\leq c'_1(\alpha^2\vee \alpha^\gamma)\left(\omega_{0,T}^{(3+\gamma)/p}
+\omega_{0,T}^{(\gamma+\gamma^2)/p}\right)\omega^{(2+\gamma)/p}_{s,t}
\leq \delta_T \varpi(\omega_{s,t}).
\end{multline*}
It remains to show that \eqref{eq:prop_lip_t_s_r} holds. As $\phi_{t,s}\in\cC_b^{1+\gamma}{(\uV,\uV)}$, thus the two semi-norms $\normlip{\phi_{t,s}}$ and
$\normsup{\vd \phi_{t,s}}$ are equivalent. We recall that 
$\phi_{t,s,r}=\phi_{t,s}\circ\phi_{s,r}-\phi_{t,r}$.
For any $a\in \uV$ and $(r,s,t)\in\rTT^3$, 
\begin{align}
\nonumber
\vd \phi_{t,s,r}(a)=& (\vd \phi_{s,r}(a)\vd f\circ\phi_{s,r}(a)-\vd f(a))\bx^{(1)}_{s,t}-\vd (\vd f\cdot f)(a)(\bx^{(2)}_{s,t}-\bx^{(1)}_{r,s}\otimes\bx^{(1)}_{s,t})\\
\nonumber
&+\vd\phi_{s,r}(a)\vd(\vd f\cdot f)\circ\phi_{s,r}(a)\bx^{(2)}_{s,t}\\
\nonumber
=&\underbrace{\left(-\vd f(a)+\vd\phi_{s,r}(a)\vd f\circ\phi_{s,r}(a)-\vd
 (\vd f\cdot f)(a)\bx_{r,s}^1\right)\bx^{(1)}_{s,t}}_{\run''_{r,s,t}}\\
&-\underbrace{\left(\vd (\vd f\cdot f)(a)-\vd \phi_{s,r}(a)\vd (\vd f\cdot f)\circ\phi_{s,r}(a)\right)\bx_{s,t}^2}_{\rdeux''_{r,s,t}}.
\label{eq:d_phi_rough}
\end{align}
Each term is estimated separately. For the first one,
\begin{align}
\nonumber
|\run''_{r,s,t}|\leq & |\vd f\circ\phi_{s,r}(a)-\vd f(a)-\vd^2 f(a) (\phi_{s,r}(a)-a)||\bx^{(1)}_{s,t}|\\
\nonumber
&+|\vd f(a)\vd f(a)\bx^{(2)}_{r,s}\bx_{s,t}^1|
+|\vd (\vd f \cdot f)(a)\vd f\circ\phi_{s,r}(a)\bx_{r,s}^{(2)}\bx^{(1)}_{s,t}|\\
\nonumber
&+|\vd f(a)[\vd f\circ\phi_{s,r}(a)-\vd f(a)]\bx_{r,s}^1\otimes\bx_{s,t}^1\\
\leq & \normf{\gamma}{\vd^2 f}|\phi_{s,r}(a)-a|^{1+\gamma}|\bx^{(1)}_{s,t}|
\nonumber
+(\normsup{\vd f}^2+\normsup{\vd(\vd f\cdot f)\vd f})|\bx_{r,s}^{(2)}||\bx^{(1)}_{s,t}|\\
\nonumber
&+\normsup{\vd f}\normsup{\vd^2 f}|\phi_{s,r}(a)-a||\bx_{r,s}^1||
\bx_{s,t}^1|\\
\leq & c_1''(\omega_{r,t}^{(2+\gamma)/p}+2\omega^{3/p}_{r,t})
\leq c_1''(1+2\omega_{0,T}^{(1-\gamma)/p})\omega_{r,t}^{(2+\gamma)/p},
\label{eq:run''}
\end{align}
where $c_1''$ is a constant which depends on $f$, $\bx$, $\gamma$.

The second one is more simple,
\begin{align}
\nonumber
|\rdeux''_{r,s,t}|\leq & |\vd(\vd f\cdot f)\circ\phi_{s,r}(a)-\vd(\vd f\cdot f)(a)||\bx^{(2)}_{s,t}|+|\vd \phi_{s,r}(a)-1|\vd(\vd f\cdot f)\circ\phi_{s,r}(a)\bx_{s,t}^2\\
\nonumber
\leq & \normsup{\vd(\vd f\cdot f)}|\phi_{s,r}(a)-a|^\gamma|\bx^{(2)}_{s,t}|
+\normsup{\vd(\vd f\cdot f)}|\vd\phi_{s,r}(a)-1||\bx^{(2)}_{s,t}|\\
\leq & c''_2 (\omega_{r,t}^{(2+\gamma)/p}+\omega_{r,t}^{3/p})
\leq  c''_2(1+\omega_{0,T}^{(1-\gamma)/p})\omega_{r,t}^{(2+\gamma)/p}.
\label{eq:rdeux''}
\end{align}
Finally, combining \eqref{eq:d_phi_rough}, \eqref{eq:run''} and \eqref{eq:rdeux''}, $\normsup{\vd\phi_{t,s,r}}\leq (c''_1+c_2'')(1+\omega_{0,T}^{(1-\gamma)/p})\varpi(\omega_{r,t})$ with
$\varpi(\omega_{r,t})=\omega_{r,t}^{(2+\gamma)/p}$. This proves
\eqref{eq:prop_lip_t_s_r} and that $\phi$ is a stable almost flow.
\end{proof}

Combining Lemma~\ref{lem:davie_almost_stable_flow}, Proposition~\ref{prop:unique_flow},
Corollaries~\ref{prop:stable_almost_flow_unique_flow} and~\ref{cor:manifold}  
leads to the following result.

\begin{corollary}
\label{cor:davie_flow}
    If $f$ is of class $\cC^{2+\gamma}_b{(\uV,L(\uU,\uV))}$, then $\phi^\pi$ converges to a unique Lipschitz flow~$\psi$. Moreover, if $\chi$ is an almost flow in a galaxy containing $\phi$, 
    then, $\chi^\pi$ converges to $\psi$. Besides, there exists a unique
    manifold of solutions to  $\vd y=\phi_{\vd t}(y)$ which is Lipschitz.
\end{corollary}


\subsection{Almost flows constructed from sub-Riemannian geodesics, as in P.~Friz and N. Victoir}
\label{subsec:friz_victoir}

In \cite{friz,friz2008}, P. Friz and N. Victoir proposed an approach based
on the use of geodesics. The following proposition is one of the fundamental result of their framework.

Now, we assume that $\uU=\RR^\ell$.

\begin{proposition}[{Remark~10.10, \cite[p.~216]{friz}}] 
    \label{prop:geodesic}
 Let $p\geq 1$ a real number and an integer $N\geq \lfloor p\rfloor$. For any $\bx\in\cpC([0,T],G^N(\RR^\ell))$ and any $(s,t)\in\TT^2$, 
    there exists a path $x^{s,t}\in\cbV(\RR^\ell)$ defined on $[s,t]$ such that 
    \begin{equation*}
	S_N(x^{s,t})_{s,t}=\bx_{s,t}
	\text{ and }
	\norma{[s,t]}{x^{s,t}}\leq K\normp{\bx}\omega_{s,t}^{1/p}\leq K'\normp{S_{\lfloor p\rfloor}(\bx)}\omega_{s,t}^{1/p}.
    \end{equation*}
    for some universal constant $K$ (resp. $K'$) that depends only on $p$ ($p$ and $N$). We say that $x^{s,t}$ is a \emph{geodesic path} associated to $\bx$.
\end{proposition}

\begin{remark}
If $x\in\cbV([s,t],\RR^\ell)$, then $\norma{[s,t]}{x}=\int_s^t\abs{\vd x_r}$.
\end{remark}

For notational convenience, we prefer now to express differential equations
with respect to vector fields, that is a family of functions 
$\hf\eqdef (\hf_1,\dots,\hf_\ell)$ that acts on $\cCb^1(\uV,L(\uU,\uV))$.
Therefore for $x\in\cbV(\RR_+,\RR^\ell)$, 
the equation $z_t=a+\int_0^t \hf\id(z_s)\vd x_s$ 
is equivalent to $y_t=a+\int_0^t f(z_s)\vd x_s$ with $f=\hf\id$. 
For a multi-indice $I\eqdef (i_1,\dots,i_k)\in\{1,\dots,\ell\}^k$ and $(s,t)\in\TT$ we
denote $\hf_I\eqdef \hf_{i_1}\dots \hf_{i_n}$. By convention, $\hf_{\emptyset}\id\eqdef \id$.


Let us fix $n\geq 2$ in $\hf\id\in\cCb^{\lambda-1}$ for $\lambda\geq n$. 
Let  $\bx$ be a rough path with values in $\uT_n(\uU)$
and  $V$ be a vector field such that $\hf\id\in\cCb^{n-1}$.
Let us define
\begin{align}
  \label{eq:phi^n}
    \phi^{(n)}_{t,s}[\bx,\hf](a)\eqdef 
    \sum_{k=0}^n \sum_{\abs{I}=k}
    \hf_I\id(a)\bx^{I}_{s,t},\ \forall (s,t)\in\rTT^2.
\end{align}

The next proposition summarizes various results on RDEs
(Theorem~10.26 in \cite[p.~233]{friz}, Theorem 10.30 in \cite[p.~238]{friz}).
When $\hf\id\in\cCb^{1+\gamma}$ but no in $\cCb^{2+\gamma}$ with $2+\gamma>p$, 
several solutions to the RDE~\eqref{eq:rde:1} may exist
(See Example~2 in~\cite{davie05a}). 

\begin{theorem}
    \label{thm:fv:1}
    Assume that $\uU$ and $\uV$ are finite dimensional Banach spaces.
    Choose $\lambda>2$ as well as an integer $n$ with $2\leq n\leq \lfloor \lambda\rfloor$.
    Let $\bx$ be a $p$-rough paths with values in $\uT_n(\uV)$.
    Let us assume that $\hf\id\in\cC^{\lambda-1}$
    for a vector field $\hf:\uV\to L(\uU,\uV)$. It holds that
    \begin{enumerate}[thm]
	\item\label{prop:fv:1:i} When $\lambda>p$,  there exists a 
	    flow $\psi[\bx,\hf]$ in the same galaxy as $\phi^{(n)}[\bx,\hf]$, 

	\item\label{prop:fv:1:iii} When $1+\lambda>p$, then there exists a 
	    unique flow as well as a unique Lipschitz manifold 
	    of solutions to  $\vd y=\phi^{(n)}_{\vd t}(y)$ ($\phi^{(n)}$
            is defined in \eqref{eq:phi^n}).

	    In addition, for any partition $\pi=\Set{t_i}_{i=0}^k$, 
	    \begin{equation}
		\label{eq:new:4}
		\abs{\psi_{t,s}[\bx,\hf](a)-{\phi^{(n)}}^\pi_{t,s}[\bx,\hf]}
		\leq
		C\normp{\bx}\sup_{i=0,\dotsc,k-1}\omega_{t_i,t_{i+1}}^{\frac{n+\gamma}{p}-1},
	    \end{equation}
	    with $\gamma\eqdef \min\Set{\lambda-n,1}$ and a constant $C$
	    that depends on $\omega_{0,T}$, $\normp{\bx}$ and $\hf$.
    \end{enumerate}
\end{theorem}

From the next lemma, we obtain that if $\hf\id\in\cCb^{\lambda}$ with $\lambda>p$, 
then there exists a unique flow associated to $\phi^{(n)}[\bx,\hf]$.

\begin{lemma} 
    For any $n\geq 2$, $\phi^{(n)}[\bx,\hf]$ belongs to 
    the same galaxy as $\phi^{(2)}[\bx,\hf]$, with 
    the Davie expansion given by \eqref{eq:rde:2} with $f=\hf\id$.
\end{lemma}
\begin{proof}
Let us write for $n\geq 2$,
\begin{equation*}
    R_{s,t}^{(n)}[\bx,\hf]\eqdef \phi_{s,t}^{(n)}[\bx,\hf]-\phi_{s,t}^{(2)}[\bx,\hf]=
    \sum_{k=3}^n \sum_{\abs{I}\leq k}
    \hf_I\id(a)\bx^{I}_{s,t},\ \forall (s,t)\in\rTT^2.
\end{equation*}
Clearly, $R^{(n)}[\bx,\hf]$ is a perturbation with $\varpi(\tau)=\tau^{3/p}$. 
Moreover, as $\hf\id\in\cCb^{\lambda-1}$, 
$\hf_I\id\in\cCb^{\lambda-k}$ for any word $I$ with $\abs{I}=k$ so 
that when $\lambda-n\geq 1$,
$R^{(n)}[\bx,\hf]$ is a Lipschitz perturbation (Definition~\ref{prop:pert:1}).
\end{proof}

This is however not sufficient to obtain the rate in~\eqref{eq:new:4}.

For a path $x\in\cbV$, we denote for $a\in\uV$ by 
$\psi_{\cdot,s}[x,\hf](a)$ the unique solution to 
\begin{equation}
    \label{eq:new:1}
\psi_{t,s}[x,\hf](a)=a+\int_s^t \hf\id(\psi_{r,s}[x,\hf](a))\vd x_r,\ t\geq s.
\end{equation}
The family $\psi[x,\hf]$ satisfies the flow property.

We set for any $(s,t)\in\rTT^2$, 
\begin{align*}
    \phi^{(n)}_{t,s}[x,\hf](a)&\eqdef \phi^{(n)}_{t,s}[S_n(x),\hf](a),
    \\
\epsilon^{(n)}_{t,s}[x,\hf](a)
&=\sum_{\abs{I}=n}
    \int_{s\leq t_{k-1}\leq\dots\leq t_1\leq t}
    \left(
	\hf_I\id(\psi_{t_k,s}[x,\hf](a))-\hf_I\id(a)
    \right)\vd x^I_{t_k},
\end{align*}
where $\vd x^I_{t_k}=\vd x_{t_k}^{i_k}\dots\vd x_{t_1}^{i_1}$.
Using iteratively the Newton formula on \eqref{eq:new:1}, 
\begin{equation*}
    \psi_{t,s}[x,\hf](a)=\phi_{t,s}^{(n)}[x,\hf](a)
    +\epsilon^{(n)}_{t,s}[x,\hf](a).
\end{equation*}

We denote by $\normhold{\gamma}{f}$ the $\gamma$-Hölder semi-norm of a function $f$.
For $\gamma=1$, this is the Lipschitz semi-norm.
Thus, $\gamma\eqdef \min\Set{\lambda-n,1}$ is the Hölder indice
of $\hf_I\id$ with $\abs{I}=n$.

From Proposition~10.3 in \cite[p.~213]{friz}, 
for a constant $C$ that depends on $\lambda$ and on 
\begin{equation*}
    \simplenorm{\hf}_*\eqdef 
\max_{\abs{I}\leq n}\normsup{\hf_I\id}
+
\max_{\abs{I}\leq n-1}\normlip{\hf_I\id}
+\max_{\abs{I}=n}\normhold{\gamma}{\hf_I\id},
\end{equation*}
it holds that 
\begin{equation}
\label{eq:new:5}
\abs{\epsilon_{t,s}^{(n)}[x,\hf](a)}
	\leq
	C\normf{[s,t],1}{x}^{n+\gamma}.
\end{equation}

\begin{lemma} 
    \label{lem:new:1}
    Let $x$ be a path of finite $1$-variation
    with $\normf{[s,t],1}{x}\leq A\omega_{s,t}$ for any $(s,t)\in\rTT^2$. 
    Let $x^{s,t}$ be the geodesic path given by Proposition~\ref{prop:geodesic}. 
    Assume that there exists a constant $K$ such that 
    $\normf{[s,t],1}{x^{s,t}}\leq K\omega_{s,t}^{1/p}$.
    Then there exists a time $T>0$ small enough and a constant $D$ that 
    depend only on $\omega$, $K$ and $\simplenorm{\hf\id}_\star$ such that
    \begin{equation*}
	\normsup{\psi_{t,s}[x^{s,t},\hf]-\psi_{t,s}[x,\hf]}
	\leq D\omega_{s,t}^{\frac{n+\gamma}{p}},\ \forall (s,t)\in\rTT^2.
    \end{equation*}
    In particular, the choice of $T$ and $D$ does \textit{not} depend on $A$.
\end{lemma}

\begin{proof}
Let us assume that $\normf{[s,t],1}{x}\leq A\omega_{s,t}$
and $\normf{[s,t],1}{x^{s,t}}\leq K\omega_{s,t}^{1/p}$.

Let $x^{r,s,t}$ be the concatenation of $x^{r,s}$ and $x^{s,t}$.
Then
\begin{multline*}
    \psi_{t,r}[x,\hf](a)-\psi_{t,r}[x^{r,t},\hf](a)
    =
    \underbrace{
    \psi_{t,s}[x,\hf](\psi_{s,r}[x,\hf](a))
    -\psi_{t,r}[x^{r,s,t},\hf](a)
    }_{\run_{r,s,t}}
    \\
    \underbrace{
    +\psi_{t,r}[x^{r,s,t},\hf](a)
    -\psi_{t,r}[x^{r,t},\hf](a)
    }_{\rdeux_{r,s,t}}.
\end{multline*}
Since $x^{r,s,t}$ is the concatenation between two paths, 
\begin{equation*}
    \psi_{t,r}[x^{r,s,t},\hf](a)=\psi_{t,s}[x^{s,t},\hf](\psi_{s,r}[x^{r,s},\hf](a)).
\end{equation*}
Thus, 
\begin{multline*}
    \abs{\run_{r,s,t}}
    \leq 
    \abs{\psi_{t,s}[x,\hf](\psi_{s,r}[x,\hf](a))
    -\psi_{t,s}[x^{s,t},\hf](\psi_{s,r}[x,\hf](a))}
    \\
    +\abs{\psi_{t,s}[x^{s,t},\hf](\psi_{s,r}[x,\hf](a))
    -\psi_{t,s}[x^{s,t},\hf](\psi_{s,r}[x^{r,s},\hf](a))}.
\end{multline*}
Writing $U_{r,t}\eqdef \normsup{\psi_{t,r}[x,\hf]-\psi_{t,r}[x^{r,t},\hf]}$,
it holds that 
\begin{equation*}
    \abs{\run_{s,r,t}}
    \leq U_{t,s}+\normlip{\psi_{t,s}[x^{s,t},\hf]}U_{r,s}.
  \end{equation*}
  From (\ref{eq:new:1}), we derive that for $t\geq s$,
  \begin{align}
    \nonumber
    \normlip{\psi_{t,s}[x^{s,t},\hf]}&\leq
    1+\normlip{\hf\id}\int_s^t\normlip{\psi_{r,s}[x^{s,t},\hf]}\abs{\vd
                                     x^{s,t}_r}\\
    \label{eq:gronwall:psi}
    &\leq 1+\normlip{\hf\id}\int_s^t\normlip{\psi_{r,s}[x^{s,t},\hf]}\abs{
    \dot{x}^{s,t}_r}\vd r,
  \end{align}
  where the derivative $\dot{x}^{s,t}$ is almost everywhere defined
  because $x^{s,t}\in\cbV(\RR^\ell)$.

  Then, using the Grönwall's inequality with
  (\ref{eq:gronwall:psi}) and Proposition~\ref{prop:geodesic}, there is constant $C$ that depends only on $K$ (defined in
Proposition~\ref{prop:geodesic}), $\normp{\bx}$  and $\normlip{\hf\id}$
such that 
\begin{equation*}
    \normlip{\psi_{t,s}[x^{s,t},\hf]}\leq \exp(C\omega_{0,T}^{1/p}).
\end{equation*}
Besides, 
\begin{equation*}
S_n(x^{r,s,t})_{r,t}
=S_n(x^{r,s})_{r,s}\otimes S_n(x^{s,t})_{s,t} 
=\bx_{r,s}\otimes\bx_{s,t}
=\bx_{r,t}
=S_n(x^{r,t})_{r,t}.
\end{equation*}
It follows that $\phi^{(n)}[x^{r,s,t},\hf]=\phi^{(n)}[x^{r,t},\hf]$. Thus, 
\begin{multline*}
\abs{\rdeux_{r,s,t}}=\abs{\epsilon^{(n)}_{t,r}[x^{r,s,t},\hf](a)-\epsilon^{(n)}_{t,r}[x^{r,t},\hf](a)}
\\
\leq            \normsup{\epsilon^{(n)}_{t,r}[x^{r,s,t},\hf]}+\normsup{\epsilon^{(n)}_{t,r}[x^{r,t},\hf]}
\leq C'\omega_{s,t}^{\frac{n+\gamma}{p}},
\end{multline*}
 where $C'\geq0$ is a new constant and
using (\ref{eq:new:5}) and Proposition~\ref{prop:geodesic} for the
last estimate.

Thus, 
\begin{equation*}
    U_{r,t}\leq U_{s,t}+\exp(C\omega_{0,T}^{1/p}) U_{r,s}
    +C'\omega_{s,t}^{\frac{n+\gamma}{p}}.
\end{equation*}

On the other hand, when $\omega_{s,t}\leq 1$, 
\begin{multline*}
    U_{s,t}= \normsup{\psi_{t,s}[x,\hf]-\psi_{t,s}[x^{s,t},\hf]}\\
\leq \normsup{\psi_{t,s}[x,\hf]-\phi^{(n)}_{t,s}[x,\hf]}
+\normsup{\psi_{t,s}[x^{s,t},\hf]-\phi^{(n)}_{t,s}[x^{s,t},\hf]}\\
\leq C''\max\Set{A^{n+\gamma}\omega_{s,t}^{n+\gamma},K^{\frac{n+\gamma}{p}}\omega_{s,t}^{\frac{n+\gamma}{p}}}\\
\leq B\omega_{s,t}^{\frac{n+\gamma}{p}}
\text{ with }B\eqdef C''\max\Set{A^{n+\gamma},K^{\frac{n+\gamma}{p}}}.
\end{multline*}
From Lemma~\ref{lem:davie:CT},
there exists a constant $D$ that does not depend on $B$ (hence on $A$) such that 
$U_{s,t}\leq D\omega_{s,t}^{\frac{n+\gamma}{p}}$.
\end{proof}

\begin{lemma}
    \label{lem:new:2}
    Let $x\in\cbV$ be as in Lemma~\ref{lem:new:2}. Then 
    \begin{equation*}
	\normsup{\epsilon^{(n)}_{t,s}[x,\hf]}=
	\normsup{\psi_{t,s}[x,\hf]-\phi^{(n)}_{t,s}[x,\hf]}
	\leq E\omega_{s,t}^{\frac{n+\gamma}{p}},
    \end{equation*}
    where $E$ depends only on $\omega$, $T$, and $K$.
\end{lemma}
\begin{proof}
    Since $\phi^{(n)}_{t,s}[x,\hf]=\phi^{(n)}_{t,s}[x^{s,t},\hf]$, 
    \begin{equation*}
	\epsilon^{(n)}_{t,s}[x,\hf]\eqdef \psi_{t,s}[x,\hf]-\phi^{(n)}_{t,s}[x,\hf]
	=\psi_{t,s}[x,\hf]-\psi_{t,s}^{(n)}[x,\hf]
	+\epsilon_{t,s}^{(n)}[x,\hf].
    \end{equation*}
    The results is then an immediate consequence of \eqref{eq:new:5}
    and Lemma~\ref{lem:new:1}.
\end{proof}

\begin{proof}[Proof of Theorem~\ref{thm:fv:1}]
  We recall that we assume that $\uU$ and $\uV$ are finite dimensional
  Banach spaces. 
For any $(a,b)\in\uV$,  any $(s,t)\in\rTT^2$, 
\begin{multline*}
    \abs{\phi^{(n)}_{t,s}[x,\hf](a)-\phi^{(n)}_{t,s}[x,\hf](b)}
    \\
    \leq
    \sum_{k=0}^{n-1}\sum_{\abs{I}=k} \normlip{\hf_I\id}\abs{x_{s,t}^I}\cdot\abs{a-b}
    +\sum_{\abs{I}=n}
    \normf{\gamma}{\hf_I\id}\cdot\abs{x_{s,t}^I}\abs{a-b}^\gamma
    \\
    \leq
    \sum_{k=0}^{n-1}\normlip{\hf_k\id}\normf{[s,t],1}{x}^k\cdot\abs{a-b}
    +
    \normf{\gamma}{\hf_n\id}\normf{[s,t],1}{x}^n\abs{a-b}^\gamma.
\end{multline*}
It then follows from Proposition~\ref{prop:perturbation_flow}
that $\phi^{(n)}[x,\hf]$ is an almost flow with 
\begin{equation*}
    \normsup*{\phi^{(n)}_{t,s}[x,\hf](\phi^{(n)}_{s,r}[x,\hf](a))
    -\phi^{(n)}_{t,r}[x,\hf](a)}
    \leq L\omega_{s,t}^{\frac{n+\gamma}{p}}, \forall (r,s,t)\in\rTT^3,
\end{equation*}
for a constant $L$ that depends only on $\normf{[s,t],1}{x}$
and of $\simplenorm{\hf}_*$.

Let $(x^m)_{m\in\NN}$ be a sequence of bounded variation paths such that
$S_{n}(x^m)\xrightarrow[m\to\infty]{} \bx$ uniformly on $[0,T]$ and such that
$\sup_{m\in\NN}\normp{S_{n}(x^m)}\leq c\normp{\bx}$ for a
uniform constant $c$ in $m$. Such a sequence exists according
to Remark~10.32 in \cite{friz}. It consists in concatenating the 
geodesic approximations given by Proposition~\ref{prop:geodesic}. 
From this, $\abs{S_{n}(x^m)_{s,t}}\leq K\omega_{s,t}^{1/p}$
with $K=c\normp{\bx}$.

Clearly, $\phi^{(n)}_{t,s}[x^m,\hf](a)$ converges to 
$\phi^{(n)}_{t,s}[\bx,\hf]$ for any $(s,t)\in\rTT^2$ and any $a\in\uV$. 
The result follows from Corollary~\ref{cor:stability}
and Lemma~\ref{lem:new:2}.
\end{proof}


\subsection{Bailleul's approach}

The approach from I.~Bailleul is the most general one. 
As it relies on using sub-Riemaniann geodesics, Friz-Victoir's approach
is limited to finite dimension, although it has recently been extedend for some parts on infinite
dimensional setting~\cite{grong}. The approach from A.M.~Davie cannot be extended to manifolds
because it involves second-order derivatives. The use of log-signature allows one to 
work on the tangent space, hence Bailleul's approach can be used on Banach manifolds.

\subsubsection{Classical control of ODE solutions}

For $a\in\uV$, the solution to the ordinary differential equation
\begin{equation*}
    y_t(a)=a+\int_0^t \hf\id(y_s(a))\vd s
\end{equation*}
is a path from $[0,T]$ to $\uV$ such that 
\begin{gather}
    \label{eq:ode:1}
    \phi(y_t(a))=\phi(a)+\int_0^t \hf\phi(y_s(a))\vd s
\end{gather}
for any $\phi\in\cC^1(\uV,\uV)$. Assuming enough regularity on both $\phi$ and $\hf$, 
we iterate~\eqref{eq:ode:1} so that 
\begin{gather*}
       \phi(y_t(a))=\phi(a)+t\hf\phi(a)+\dotsb+\frac{t^k}{k!}\hf^{\,k}\phi(a)
    +R(\hf^{\,k}\phi,a;t)
\end{gather*}
with $\hf^{\,0}\phi=\phi$, $\hf^{\,k+1}\phi=\hf(\hf^{\,k}\phi)$, $k=0,1,\dotsc$ and 
\begin{equation*}
    R_k(\psi,a;t)=\int_0^t \int_0^{t_1}\dotsb \int_0^{t_{k-1}}
    \big(\psi(y_{t_k}(a))-\psi(a)\big)\vd t_{k}\dotsb\vd t_1.
\end{equation*}
for a function $\psi:\uV\to\uV$.

\begin{lemma}
\label{lem:ode}
    If $\hf\id$ is uniformly Lipschitz, then for any $a\in\uV$ and any $t\geq 0$.  
    \begin{gather}
	\label{eq:ode:2}
	\abs{y_t(a)-a}\leq t \abs{\hf\id(a)}\exp(\normlip{\hf\id}t).
    \end{gather}
    Moreover, if $\hf\id$ satisfies a 4-points control, then for any $a,b\in\uV$ and any $t\geq 0$,  
    \begin{multline}
	\label{eq:ode:5}
	\Delta_t(a,b)\eqdef \abs{y_s(a)-a-y_s(b)+b}\\
	\leq
	t\widehat{\hf\id}(\alpha_t(a,b))\exp\left((\widehat{\hf\id}(\alpha_t(a,b))+\oldwidecheck{(\hf\id)})t\right)\abs{a-b}
    \end{multline}
    with 
    \begin{equation}
	\label{eq:ode:6}
	\alpha_t(a,b)\eqdef \sup_{s\in[0,t]}\abs{y_s(a)-a}\vee\abs{y_s(b)-b}
	\leq t\big((\abs{\hf\id(a)}\vee\abs{\hf\id(b)})\exp(\normlip{\hf\id}t)\big).
    \end{equation}
    In particular, if $\hf\id$ satisfies a 4-points control and is bounded, then $a\mapsto y(a)$
    is Lipschitz from $\uV$ to $(\cC([0,T],\uV),\normsup{\cdot})$.
\end{lemma}
\begin{proof} Let us write $v\eqdef \hf\id\in\cC^1(\uV,\uV)$. 
    Since
    \begin{equation}
	\label{eq:ode:3}
	y_t(a)-a=\int_0^t (v(y_s(a))-v(a))\vd s+t v(a), \text{ for any }t\geq 0, 
    \end{equation}
    an immediate application of the Gronwall lemma gives \eqref{eq:ode:2}.

    Since $v$ satisfies a 4-points control, for $a,b\in\uV$, with 
    $\Delta_s(a,b)\eqdef \abs{y_s(a)-a-y_s(b)+b}$, 
    \begin{multline}
	\label{eq:ode:4}
	\abs{v(y_s(a))-v(a)-v(y_s(b))+v(b)}\\
	\leq \widehat{v}(\abs{y_s(a)-a}\vee\abs{y_s(b)-b})\abs{y_s(a)-y_s(b)}\vee \abs{a-b}
	+\oldwidecheck{v}\Delta_s(a,b).
    \end{multline}
    Besides, 
    \begin{equation*}
	\abs{y_s(a)-y_s(b)}\leq \abs{a-b}+\Delta_s(a,b).
    \end{equation*}
    Injecting \eqref{eq:ode:4} into \eqref{eq:ode:3} shows that
    \begin{equation*}
	\Delta_t(a,b)
	\leq  
	\alpha_t(a,b) t\widehat{v}(\alpha_t(a,b))\abs{a-b}
	+(\widehat{v}(\alpha_t(a,b))+\oldwidecheck{v})\int_0^t \Delta_s(a,b)\vd s
    \end{equation*}
    with $\alpha_t(a,b)$ given by \eqref{eq:ode:6}.
    Again, the Gronwall lemma yields~\eqref{eq:ode:5}.
\end{proof}


\subsubsection{Bailleul's approach by truncated logarithmic series}

Here $\uU=\RR^\ell$ for a dimension $\ell\geq 1$.

Let 
${\hf}\eqdef (\hf_1,\dots,\hf_\ell)$ be a family of vector fields which acts on $C^1(\uV,\uV)$ and
  $\bx\in\cG^{p}(\RR^\ell))$ be a weak-geometric $p$-rough path with $2\leq p<3$. By definition of the weak geometric rough paths, $\bx^{i,j}+\bx^{j,i}=\bx^{i}\bx^{j}$
  for any $i,j\in\Set{1,\dotsc,\ell}$.
We denote by $[\hf_i,\hf_j]\eqdef \hf_i\hf_j-\hf_j\hf_i$, the Lie bracket of vector fields
$\hf_i$ and $\hf_j$. The Lie bracket is itself a vector field.

Assuming that {$\hf$} is smooth, we define for any $(s,t)\in\TT^2$,
$\alpha\in\TT$ and $a\in\uV$, the solution $(\alpha,a)\mapsto y_{s,t}(\alpha,a)$ of the ODE,
\begin{align}
y_{s,t}(\alpha,a)=a+\int_0^\alpha \hf_i\id(y_{s,t}(\beta,a))\bx^{i}_{s,t}\vd\beta+
\frac{1}{2}\int_0^\alpha [\hf_i,\hf_j]\id(y_{s,t}(\beta,a))\bx^{i,j}_{s,t}\vd\beta,
\label{eq:ode_bailleul}
\end{align} 
where we omit the summation over all indice $i,j\in\{1,\dots,\ell\}$. 
We write 
\begin{equation}
    \label{eq:bailleul:1}
\chi_{t,s}(a)\eqdef y_{s,t}(1,a)
=\phi_{s,t}+ \epsilon_{t,s},
\end{equation}
where, by iterating \eqref{eq:ode_bailleul},
\begin{align*}
\phi_{t,s}(a)&\eqdef a+\hf_i\id(a)\bx^{i}_{s,t}+\frac{1}{2}\hf_i\hf_j\id(a)\bx^{i}_{s,t}\bx^{j}_{s,t}+\frac{1}{2}[\hf_i,\hf_j]\id(a)\bx^{i,j}_{s,t},\\
\epsilon_{t,s}(a)&\eqdef \int_0^1\int_0^\beta \hf_i\hf_j(\id(y_{s,t}(\gamma,a))-\id(a)]\bx^{i}_{s,t}\bx^{j}_{s,t}\vd\gamma\vd\beta\\
&+
\frac{1}{2}\int_0^1\int_0^\beta[\hf_i,\hf_j][\id(y_{s,t}(\gamma,a))-\id(a)]\bx^{i,j}_{s,t}\vd\gamma\vd\beta\\
&+
\frac{1}{2}\int_0^1\int_0^\beta \hf_i[\hf_j,\hf_k]\id (y_{s,t}(\gamma,a))\bx^{i}_{s,t}\bx^{j,k}_{s,t}\vd\gamma\vd\beta.
\end{align*}
With the weak geometric property of $\bx$ : $\bx_{s,t}^{i,j}+\bx^{j,i}_{s,t}=x^i_{s,t}x^j_{s,t}$, we simplify the expression of $\phi$
such that
\begin{align*}
\phi_{t,s}(a)&=a+\hf_i\id(a)\bx^{i}_{s,t}+\frac{1}{2}\hf_i\hf_j\id(a)\left(\bx^{i,j}_{s,t}+\bx^{j,i}_{s,t}\right)+\frac{1}{2}[\hf_i,\hf_j]\id(a)\bx^{i,j}_{s,t}\\
&=a+\hf_i\id(a)\bx^{i}_{s,t}+\hf_i\hf_j\id(a)\bx^{i,j}_{s,t}.
\end{align*}
So $\phi$ corresponds to the Davie's almost flow defined in \eqref{eq:rde:2}.

\begin{proposition}
    Assume that ${\hf\id}\in\cC^{2+\gamma}_b$ with $2+\gamma>p$. Then,
$\chi$ defined by~\eqref{eq:bailleul:1} is an almost flow which generates a Lipschitz manifold of solutions.
Moreover $\chi^\pi$ converges to the Davie's flow $\psi$ of the Corollary~\ref{cor:davie_flow}.
\end{proposition}

\begin{proof}
We proved in Lemma~\ref{lem:davie_almost_stable_flow} that $\phi$ is a stable
almost flow. We shall show that~$\epsilon_{t,s}$ is a perturbation in the sense
of Definition~\ref{def:perturb} and then we use Proposition~\ref{prop:pert:1}
to conclude that $\chi$ is an almost flow which is in the galaxy of $\phi$.  We
use Corollary~\ref{cor:davie_flow} and Remark~\ref{rem:unique_flow} to conclude
the proof.
 
It is straightforward that $\epsilon_{t,t}=0$.
We start by computing an \textit{a priori} estimate of $(\alpha,a)\mapsto y_{s,t}(\alpha,a)$, for any $(s,t)\in\TT^2$,
$a\in\uV$ and $\alpha\in [0,1]$,
\begin{align}
\nonumber
\abs{y_{s,t}(\alpha,a)-a}&\leq \alpha\normsup{\hf_i\id}\normp{\bx^i}\omega_{s,t}^{1/p}
+\frac{\alpha}{2}\normsup{[\hf_i,\hf_j]\id}\normpp{\bx^{i,j}}\omega_{s,t}^{2/p}\\
&\leq \left(\normsup{\hf_i\id}+\normsup{[\hf_i,\hf_j]\id}
\omega_{0,T}^{1/p}\right)\normp{\bx}\omega_{s,t}^{1/p}\\
&\leq C_\infty\normp{\bx}\omega_{s,t}^{1/p},
\label{eq:apriori_bailleul}
\end{align}
where $C_\infty\eqdef \normsup{\hf_i\id}+\normsup{[\hf_i,\hf_j]\id}
\omega_{0,T}^{1/p}$.
With \eqref{eq:apriori_bailleul}, we control the remainder $\epsilon_{t,s}$,
\begin{align*}
\normsup{\epsilon_{t,s}}&\leq \left[ \normp{\bx^{i}}\normp{\bx^j}\normlip{\hf_i\hf_j\id}
+\normf{2p}{\bx^{i,j}}\normlip{[\hf_i,\hf_j]\id}\right]
C_\infty
\normp{\bx}\omega_{s,t}^{3/p}\\
&+\normsup{\hf_i[\hf_j,\hf_k]\id}\normp{\bx^{i}}\normf{2p}{\bx^{j,k}}\omega_{s,t}^{3/p},
\end{align*}
which proves \eqref{eq:epsilon:2}.

To show the last estimate~\eqref{eq:epsilon:3}, we compute for any $(s,t)\in\rTT^2$ and any $a,b\in\uV$,
\begin{align*}
\epsilon_{t,s}(b)-\epsilon_{t,s}(a)&=\underbrace{\int_0^1\int_0^\beta   \hf_i\hf_j(\id(y_{s,t}(\gamma,b))-\id(b)-\id(y_{s,t}(\gamma,a))+\id(a)]\bx^{i}_{s,t}\bx^{j}_{s,t}\vd\gamma\vd\beta}_ {\run}\\
&+\underbrace{\frac{1}{2}\int_0^1\int_0^\beta[\hf_i,\hf_j][\id(y_{s,t}(\gamma,b))-\id(b)-\id(y_{s,t}(\gamma,a))+\id(a)]\bx^{i,j}_{s,t}\vd\gamma\vd\beta}_{\rdeux}\\
&+
\underbrace{
\frac{1}{2}\int_0^1\int_0^\beta \hf_i[\hf_j,\hf_k][\id(y_{t,s}(\gamma,b))-\id (y_{t,s}(\gamma,a))]\bx^{i}_{s,t}\bx^{j,k}_{s,t}\vd\gamma\vd\beta}_{\rtrois}.
\end{align*}
We assume that $\hf\id\in\cC_b^{2+\gamma}$, so
$\hf_i\hf_j\id\in\cC^{1+\gamma}_b$. It follows from Lemma~\ref{lem:4_points_f}
that $\hf_i\hf_j$ satisfies a $4$-points control such that
$\widehat{\hf_i\hf_j}\id(x)=C\abs{x}^\gamma$ where $C$ a positive constant with
depends on the $\gamma$-Hölder norm of the derivative of $\hf_i\hf_j$. It
follows that,
\begin{align*}
\abs{\run}&\leq C\sup_{\gamma\in [0,1],a\in\uV}\abs{y_{s,t}(\gamma,a)-a}^\gamma\left[\sup_{\gamma\in[0,1]}\abs{y_{s,t}(\gamma,b)-y_{s,t}(\gamma,a)}+\abs{b-a}\right]\normp{\bx^{(1)}}^2\omega_{s,t}^{2/p}\\
&+\normlip{\hf_i\hf_j}\sup_{\gamma\in [0,1]}\abs{y_{s,t}(\gamma,b)-b-y_{s,t}(\gamma,a)+a}\normp{\bx^{(1)}}^2\omega_{s,t}^{2/p},
\end{align*}
which yields combining with \eqref{eq:ode:5}, \eqref{eq:ode:6} and
\eqref{eq:apriori_bailleul} to
\begin{align*}
\abs{\run}\leq CC_{\infty,T}^\gamma\normp{\bx}^\gamma\omega_{s,t}^{\gamma/p}
(1+C_{T})\abs{b-a}\normp{\bx^{(1)}}^2\omega_{s,t}^{2/p}
+\normlip{\hf_i\hf_j}C_{T}\abs{b-a}\normp{\bx^{(1)}}^2\omega_{s,t}^{2/p},
\end{align*}
where $C_T$ is a constant which is computed in \eqref{eq:ode:5}.
Finally, $\abs{\run}\leq \delta_T\abs{b-a}$ where $\delta_T$ is a constant depending on
the norms of $\hf$, $\bx$ which decreases to $0$ when $T\rightarrow 0$. 
Similarly, we obtain the same estimate for $\rdeux$.
To estimate $\rtrois$, we note that $\hf_i[\hf_j,\hf_k]\id\in\cC^{\gamma}_b$.
Then with \eqref{eq:ode:5} it follows that $\abs{\rtrois}\leq
C''_T\omega_{s,t}^{2/p}\abs{b-a}^\gamma$, where $C''_T$ is another constant
which has the same dependencies as $C'_T$.  Thus
$\abs{\epsilon(b)-\epsilon(a)}\leq
\delta\abs{b-a}+C''_T\omega_{s,t}^{2/p}\abs{b-a}^\gamma$. This concludes the
proof.
\end{proof}

\begin{remark}
This result can be extended to the case $\uU$ is a Banach case.
It is an advantage compared to the Friz-Victoir's approach of Subsection~\ref{subsec:friz_victoir}.
\end{remark}

\bigskip
\noindent
\textbf{Acknowledgement.} The authors wish to thank Laure Coutin for her careful
reading and interesting discussions regarding the content of this article.
We also thank the referees for their suggestions and comments which improved the manuscript.
The first author thanks the Center for Mathematical Modeling, Conicyt fund AFB 170001.



\printbibliography

\end{document}